\documentclass[12pt, a4paper]{amsart}
\usepackage{amsmath,amsthm,amssymb,amscd}
\usepackage[shortalphabetic]{amsrefs}
\usepackage[all]{xy}
\usepackage{empheq}
\usepackage{color}
\usepackage{hyperref}
\usepackage{graphicx}
\usepackage{epstopdf}
\usepackage{epsfig}

\newtheorem{thm}{Theorem}[section]

\newtheorem{cor}[thm]{Corollary}
\newtheorem{lem}[thm]{Lemma}
\newtheorem{pro}[thm]{Proposition}
\theoremstyle{definition}
\newtheorem{eg}[thm]{Example}
\newtheorem{de}[thm]{Definition}

\newtheorem*{re}{Remark}

\newcommand{\excise}[1]{}

\DeclareMathOperator{\length}{length}
\DeclareMathOperator{\diam}{diam}
\DeclareMathOperator{\post}{post}
\DeclareMathOperator{\inter}{int}

\def\R{\mathbb{R}}
\def\C{\mathbb{C}}

\def\Z{\mathbb{Z}}
\def\S{\mathbb{S}}

\def\N{\mathbb{N}}

\def\SL{{\rm SL}}

\def\T{\mathcal{T}}
\def\L{\mathcal{L}}
\def\D{\mathcal{D}}

\def\:{\colon}

\def\ra{\rightarrow}

\title{Latt\`es maps and combinatorial expansion}
\author{Qian Yin}
\thanks{The author was partially supported by NSF grants  DMS 0757732, DMS 0353549, DMS 0456940, DMS 0652915, DMS 1058772, and DMS 1058283.}
\address{Qian Yin \\The University of Chicago\\ 5734 S. University Avenue\\ Chicago, Illinois 60637\\USA}
\email{qyin@math.uchicago.edu}

\begin{document}
\maketitle

\begin{abstract}
A Latt\`es map $f\:\widehat\C\ra \widehat\C$ is a rational map that is obtained from a finite quotient of a conformal  torus endomorphism. We characterize Latt\`es maps by their combinatorial expansion behavior.
\end{abstract}

\tableofcontents

\section{Background}
\noindent
A \emph{rational map} $f\:\widehat\C\ra \widehat\C$ is a special type of analytic map on the Riemann sphere $\widehat\C=\C\cup \{\infty\}$. It can be written as a quotient of two relatively prime complex polynomials $p(z)$ and $q(z)$, with $q(z)\not=0$,
\begin{eqnarray}\label{rationalmap}
f(z)=\frac{p(z)}{q(z)}=
\frac{a_0z^{m}+\ldots+a_{m}}{b_0z^l+\ldots+b_l},
\end{eqnarray}
where $a_i,b_j \in \C$ for $i=0,\ldots,m$ and $j=0,\ldots, l$.
The fundamental problem in dynamics is to understand the behavior of the iterates of $f$,
\[f^n(z):=\underbrace{f\circ f\circ \cdots \circ f}_{\textstyle{n} \mbox{ factors}}(z).\]

The study of the dynamics of rational maps originated in 1917 by Pierre Fatou and Gaston Julia, who developed the foundations of complex dynamics. In particular, they applied Montel's theory of normal families to develop the fundamental theory of iteration (see \cite{FatSur} and \cite{JulMemoire}). Their work was more or less forgotten for over half a century. Then Benoit Mandelbrot rekindled interest in the field in the 1970s by generating beautiful and intriguing graphic images that naturally appear under iteration of rational maps through his computer experiments (see \cite{ManObjets} and \cite{ManFractal}). In recent years, the study of dynamics of rational maps has attracted considerable interest, not only because complex dynamics itself is an intriguing and rich subject, but also because of its links to other branches of mathematics, such as quasi-conformal mappings, Kleinian groups, potential theory and algebraic geometry. For instance, the study of the dynamical systems arising from polynomials and those that arise from Kleinian groups that depend on holomorphic motions are connected
by the dictionary introduced by Sullivan (see \cite{SulQuasiconformal1}), which led to his seminal work on the non-existence of wandering domains for rational maps.

Given a rational map $f\:\widehat\C\ra\widehat\C$, the \emph{degree} $\deg(f)$ of $f$ is the maximal degree of the polynomials $p(z)$ and $q(z)$ as in equation \eqref{rationalmap}. The degree of $f$ can also be defined topologically as the cardinality of the preimage over a generic (non-critical) value.

A rational map $f$ with $\deg(f)>1$ can have both expanding and contracting features.  The tension between these two features makes the dynamics of rational maps involved and interesting. \emph{We will assume that the rational map $f$ has $\deg(f)>1$ from now on.} A point $z\in \widehat\C$ is \emph{periodic} if $f^n(z)=z$ for some $n\geq 1$. In this case, it is called
\[\begin{array}{lc}
  \mbox{attracting} & \mbox{if } |(f^n)'(z)|<1; \\
  \mbox{indifferent} & \mbox{if } |(f^n)'(z)|=1;  \\
  \mbox{repelling} & \mbox{if } |(f^n)'(z)|>1.
\end{array}\]
For example, if we let $f(z)=z^2$, then $z=0$ is an attracting periodic point of $f$, and $f$ is contracting near $0$; $z=1$ is a repelling periodic point, and $f$ is expanding near $1$.

The \emph{Julia set} $J(f)$ of $f$ is the closure of the set of repelling periodic points. It is also the smallest closed set containing at least three points which is completely invariant under $f^{-1}$. For the example $f(z)=z^2$, the Julia set of $f$ is the unit circle. The complement $F(f)=\widehat\C\setminus J(f)$ of the Julia set, called the \emph{Fatou set}, is the largest open set such that the iterates of $f$ restricted to it form a normal family. The Julia set and Fatou set are both invariant under $f$ and $f^{-1}$.

The \emph{postcritical set} $\post(f)$ of $f$ is the forward orbits of the critical points
\begin{eqnarray*}
\post(f)=\bigcup_{n\geq 1}\{f^n(c)\: c\in {\rm crit}(f)\}.
\end{eqnarray*}
The postcritical set plays a crucial role in understanding the expanding and contracting features of a rational map.
If the postcritical set $\post(f)$ is finite, we say that the map $f$ is \emph{postcritically finite}.

\excise{
In the postcritically finite case,
\begin{eqnarray*}
\post(f)=\bigcup_{n\geq 1}\{f^n(c)\: c\in {\rm crit}(f)\}.
\end{eqnarray*}
}

In 1918, Samuel Latt\`es described a special class of rational maps which have a simultaneous linearization for all of their periodic points (see \cite{LatSur}). This class of maps is named after Latt\`es, even though similar examples had been studied by Ernst Schr\"oder  much earlier (see \cite{SchUeber}). A \emph{Latt\`es map} $f\:\widehat\C\ra \widehat\C$ is a rational map that is obtained from a finite quotient of a conformal torus endomorphism, i.e., the map $f$ satisfies the following commutative diagram:
\begin{equation}\label{lat}
    \begin{CD}
\T @>\bar{A}>> \T\\
@V\Theta VV @VV\Theta V\\
\widehat\C @>f>> \widehat\C
\end{CD}
\end{equation}
where $\bar A$ is a map of a torus $\T$ that is a quotient of an affine map of the complex plane, and $\Theta$ is a finite-to-one holomorphic map. Latt\`es maps were the first examples of rational maps whose Julia set is the whole sphere $\widehat \C$, and the postcritical set of a Latt\`es map is finite. More importantly, Latt\`es maps play a central role as exceptional examples in complex dynamics. We will discuss this further in the following section.

Observing that much information about the dynamics of a rational map can be deduced from the postcritical set,
Thurston introduced a topological analog of a postcritically finite rational map, now known as a \emph{Thurston map}.
A \emph{Thurston map} $f\:\S^2\ra \S^2$ is a branched covering map with finite postcritical set $\post(f)$. Thurston characterized Latt\`es maps among Thurston maps up to \emph{Thurston equivalence} (see Definition \ref{thurstoneq}) in terms of associated orbifolds and the derivatives of associated torus automorphisms  (see Section 9 in \cite{DHThurston}).

\excise{
Thurston gave a characterization of Latt\`es maps among Thurston maps in terms of the angles of cone points
on associated orbifolds. More precisely, for any Thurston map $f$, there exists a unique smallest function $\nu_f$ among functions $\nu\: \S^2\ra \N\cup \{\infty\}$ such that
\begin{equation}\label{divisor}
    \nu(p)\deg_f(p)\big|\nu(f(p))
\end{equation}
for all $p\in \S^2$. Thurston associated an \emph{orbifold} $O_f=(\S^2,\nu_f)$ to a Thurston map $f$
using the function $\nu_f$ (see \cite{DHThurston}). For each $p\in \S^2$ with $v_f(p)\not=1$, the point $p$ is a \emph{cone point} with \emph{cone angle} $2\pi/v_f(p)$.
Let
\[\chi(O_f)=2-\sum_{p\in \post(f)}\left(1-\frac{1}{\nu_f(p)}\right).\]
If $\chi(O_f)=0$, we say that the orbifold $O_f$ is \emph{parabolic}. The orbifold $O_f$ of a Latt\`es map is parabolic, and the cardinality of the postcritical set can only be $3$ 
or $4$. Latt\`es maps with $\#\post(f)=4$ are also referred to as \emph{flexible} Latt\`es maps.
}

The notion of an expanding Thurston map was introduced in \cite{BMExpanding} as a topological analog of a postcritically finite rational map whose Julia set is the whole sphere $\widehat{\C}$. Roughly speaking, a Thurston map is called \emph{expanding} if all the connected components of the preimage under $f^{-n}$ of any open Jordan region disjoint from $\post(f)$ become uniformly small as $n$ tends to infinity. We refer the reader to Definition~\ref{expandingmap} for a more precise statement. A related and more general notion of expanding Thurston maps was introduced in \cite{HPCoarse}. Latt\`es maps are among the simplest examples of expanding Thurston maps.

Latt\`es maps are distinguished among all rational maps in various ways. For instance,
Latt\`es maps are the only rational maps for which the measure of maximal entropy is absolutely continuous with respect to Lebesgue measure (see \cite{ZduParabolic}).

Many different characterizations of Latt\`es maps have been both given and conjectured
(e.g. \cite{MeyDimension}).
For example, a fundamental conjecture in complex dynamics states that the flexible Latt\`es maps are the only rational maps that admit an ``invariant line field'' on their Julia set. The significance of this conjecture is demonstrated by a theorem of Ma\~n\'e, Sad and Sullivan (see \cite{MSSdynamics}). It states that if the fundamental conjecture above is true, then hyperbolic maps are dense among rational maps.  We refer the reader to \cite{MilLattes} for a nice exposition on Latt\`es maps.

\section{Summary of Results}
\noindent
Let $f$ be an expanding Thurston map, and let $\mathcal C$ be a Jordan curve containing $\post(f)$.
The Jordan Curve Theorem implies that $\S^2\setminus\mathcal C$ has precisely two connected components, whose closures we call \emph{$0$-tiles}. We call the closure of each connected component of the preimage of $\S^2\setminus\mathcal C$ under $f^n$ an \emph{$n$-tile}. In Section 5 of \cite{BMExpanding}, it is proved that, for every $n\geq 0$, the collection of all $n$-tiles gives a cell decomposition of $\S^2$. The points in $\post(f)$ divide $\mathcal C$ into several subarcs.  Let $D_n=D_n(f,\mathcal{C})$ be the minimum number of $n$-tiles needed to join two of these subarcs that are non-adjacent (see Definition \ref{joinoppositesides} and \eqref{defdn}). For any Thurston map $f$ without periodic critical points, there exists $C>0$ such that
\begin{equation}\label{eq}
D_n\leq C(\deg f)^{n/2}
\end{equation}
for all $n>0$ (see Proposition \ref{other}). For Latt\`es maps, an inequality as in \eqref{eq} is also true in the opposite direction (see Proposition \ref{opinequality}). One of the main results of this paper asserts that, in fact, this inequality characterizes Latt\`es maps among expanding Thurston maps with no periodic critical points (see Theorem \ref{main}).
\begin{thm} \label{main0}
A map $f\:\S^2\ra \S^2$ is topologically conjugate to a Latt\`es map if and only if the following conditions hold:
\begin{itemize}
  \item $f$ is an expanding Thurston map;
  \item $f$ has no periodic critical points;
  \item there exists $c>0$ such that $D_n\geq c(\deg f)^{n/2}$ for all $n>0$.
\end{itemize}
\end{thm}
There is an interpretation (see Theorem 5.3 in \cite{YinThurston}) of Theorem \ref{main0} in the Sullivan dictionary corresponding to Hamenst{\"a}dt's entropy rigidity theorem (see \cite{HamEntropy}).

Let $f$ be an expanding Thurston map. Even though $D_n=D_n(f,\mathcal C)$ depends on the Jordan curve $\mathcal C$, its growth rate is independent of $\mathcal C$. Hence the limit
\begin{equation}\label{comb}
    \Lambda_0(f)=\lim_{n\ra \infty}\big(D_n(f,\mathcal C)\big)^{1/n}
\end{equation}
exists and only depends on the map $f$ itself (see \cite[Prop.~17.1]{BMExpanding}). We call this limit $\Lambda_0(f)$ the \emph{combinatorial expansion factor} of $f$. This quantity $\Lambda_0(f)$ is invariant under topological conjugacy and is multiplicative in the sense that $\Lambda_0(f)^n$ is the combinatorial expansion factor of $f^n$.
Inequality \eqref{eq} implies that
\[\Lambda_0(f)\leq (\deg f)^{1/2}.\]


The combinatorial expansion factor is closely related to the notion of \emph{visual metrics and their expansion factors}. Every expanding Thurston map $f\:\S^2\ra \S^2$ induces a natural class of metrics on $\S^2$, called \emph{visual metrics} (see Definition \ref{visual}), and each visual metric $d$ has an associated \emph{expansion factor} $\Lambda > 1$. This visual metric is essentially characterized by the geometric property that the diameter of an $n$-tile is about $\Lambda^{-n}$, and the distance between two disjoint $n$-tiles is at least about $\Lambda^{-n}$. The supremum of the expansion factors of all visual metrics is equal to the combinatorial expansion factor $\Lambda_0$ (see \cite[Theorem 1.5]{BMExpanding}). For Latt\`es maps, the supremum is obtained. In general, the supremum is not obtained. For examples, the supremum is not obtained for Latt\`es-type maps that are not Latt\`es maps (see Section \ref{lattes}).
We will show in Proposition \ref{existenceofvisualmetric} that Theorem \ref{main0} remains true if we replace the third condition by the requirement that there exists a visual metric on $\S^2$ with expansion factor $\Lambda=(\deg f)^{1/2}$.

Here we outline the sufficiency of the three conditions in Theorem \ref{main0}. These three conditions imply the existence of a visual metric $d$ on $\S^2$ with expansion factor $\Lambda=(\deg f)^{1/2}$ (see Proposition \ref{existenceofvisualmetric}). This is the most technical part of the paper. The way that we construct the visual metric uses the idea that any quasi-visual metric can be modified to be a visual metric.
The existence of this visual metric implies that $(\S^2,d)$ is Ahlfors $2$-regular, which means any ball with radius $r$ has Hausdorff $2$-measure roughly $r^2$
(see Proposition \ref{qstosphere}). Using this $2$-regularity together with the linear local connectivity  condition of $(\S^2, d)$, we obtain that $(\S^2,d)$ is quasisymmetrically equivalent to the Riemann sphere $\widehat\C$ by \cite[Theorem 1.1]{BKQuasisymmetric} (see Proposition \ref{qstosphere}).
We deduce that $f$ is topologically conjugate to a rational map from the quasisymmetrical equivalence of $(\S^2,d)$ to the Riemann sphere $\widehat \C$ (see Proposition \ref{smallsum}). Now we can focus on the rational maps with three conditions satisfied in Theorem \ref{main0}. In order to invoke the characterization of Latt\`es maps among rational maps by \cite{MeyDimension}, we need that the Hausdorff measure with respect to the visual metric and with respect to the standard chordal metric $d$ on $\widehat \C$ are essentially the same. This follows from a theorem by Juha Heinonen and Pekka Koskela 
(see Theorem \ref{absolutecontinuous}), and it implies that the dimension of Lebesgue measure with respect to the visual metric $d$ is equal to 2 (see Theorem \ref{suff}). We conclude that the map $f$ is topologically conjugate to a Latt\`es map.

We define \emph{Latt\`es-type} maps so as to include non-rational maps that are quotients of affine maps and share many desired properties of Latt\`es maps. Comparing with diagram \ref{lat},
we have a commutative diagram
\[\begin{CD}
\T @>\bar{A}>> \T\\
@V\Theta VV @VV\Theta V\\
\S^2 @>f>> \S^2
\end{CD}.\]
where $\Theta$ is essentially the same $\Theta$ as in diagram \ref{lat}, and we require $\bar A$ to be a quotient of an affine map on the real plane rather than the complex plane.
A map $f\:\S^2\ra\S^2$ obtained by the above commutative diagram is called a \emph{Latt\`es-type} map (see Definition \ref{lattetype}). If a Latt\`es-type map 
is rational, then the map 
is a Latt\`es map.

Latt\`es-type maps are examples of expanding Thurston maps, and they have the same orbifold structures as Latt\`es maps (see Proposition \ref{dnrelation}).
\begin{pro}
Let $f$ be a Latt\`es-type map with orbifold type $(2,2,2,2)$. Let $A$ be its corresponding linear map from $\R^2$ to $\R^2$ and let $\wp\: \R^2\ra \S^2$ be the Weierstrass function with the lattice $2\Z^2$. We have
\[\frac{1}{\|A^{-n}\|_{\infty}}\leq D_n(f,\mathcal{C})\leq \frac{1}{\|A^{-n}\|_{\infty}}+1,\]
where the Jordan curve $\mathcal{C}$ is the image of the boundary of the unit square $[0,1]\times[0,1]$ under $\wp$.
\end{pro}
Here $\|B\|_\infty$ denotes the operator norm of a linear map $B$ on $\R^2$ with respect to
the $\ell^{\infty}$-norm.
As a corollary of this proposition and equation \eqref{comb}, we have the following result (see Corollary \ref{combexpansionfactor}).
\begin{cor} 
Let $f$ be a Latt\`es-type map with orbifold type $(2,2,2,2)$, and let $A$ be the corresponding
linear map from $\R^2$ to $\R^2$.
Then the combinatorial expansion factor $\Lambda_0(f)$ equals the minimum absolute value of the eigenvalues of $A$.
\end{cor}

\bigskip
\noindent
\textbf{Acknowledgements.} This paper is part of the author's PhD thesis under the supervision of Mario Bonk. The author would like to thank Mario Bonk for introducing her to and teaching her about the subject of Thurston maps and its related fields. The author is inspired by his enthusiasm and mathematical wisdom, and is especially grateful for his patience and encouragement. The author would like to thank Dennis Sullivan for valuable conversations and sharing his mathematical insights.
The author also would like to thank Kyle Kinneberg, Alan Stapledon and Michael Zieve for useful comments and feedback.

\section{Expanding Thurston maps and Cell Decompositions} \label{expanding}
\noindent
In this section we review some definitions and facts on expanding Thurston maps. We refer the reader to Section 3 in \cite{BMExpanding} for more details. We write $\N$ for the set of positive integers, and $\N_0$ for the set of non-negative integers. We denote the identity map on $\S^2$ by ${\rm id}_{\S^2}$.

Let $\S^2$ be a topological 2-sphere with a fixed orientation. A continuous map $f\:\S^2\ra \S^2$ is called \emph{a branched covering map} over $\S^2$ if $f$ can be locally written as
\[z\mapsto z^d\]
under certain orientation-preserving coordinate changes of the domain and range. More precisely, we require that for any point $p\in \S^2$, there exists some integer $d>0$, an open neighborhood $U_p\subset \S^2$ of $p$, an open neighborhood $V_q\subseteq \S^2$ of $q=f(p)$, and orientation-preserving homeomorphism
\[\phi\: U_p\ra U\subseteq \C\]
and
\[\psi \: V_p\ra V\subseteq \C\]
with $\phi(p)=0$ and $\psi(q)=0$ such that
\[(\psi\circ f \circ \phi^{-1} )(z)=z^d\]
for all $z\in U$. The positive integer $d=\deg_f(p)$ is called the \emph{local degree} of $f$ at $p$ and only depends on $f$ and $p$. A point $p\in \S^2$ is called a \emph{critical point} of $f$ if $\deg_f(p)\geq 2$, and a point $q$ is called \emph{critical value} of $f$ if there is a critical point in its preimage $f^{-1}(q)$. If $f$ is a branched covering map of $\S^2$, $f$ is open and surjective. There are only finitely many critical points of $f$ and $f$ is \emph{finite-to-one} due to the compactness of $\S^2$. Hence, $f$ is a covering map away from critical values in the range and the preimages of critical values in the domain. The \emph{degree $\deg(f)$} of $f$ is the cardinality of the preimage over a non-critical value. In addition, we have
\[\deg(f)=\sum_{p\in f^{-1}(q)}\deg_f(p)\]for every $q\in \S^2$. For $n\in \N$, we denote the $n$-th iterate of $f$ as
\[f^n=\underbrace{f\circ f\circ \cdots \circ f}_{\textstyle{n} \mbox{ factors}}.\]
We also set $f^0={\rm id}_{\S^2}$. If $f$ is a branched cover of $\S^2$, so is $f^n$, and
\[\deg(f^n)=\deg(f)^n.\]Let crit$(f)$ be the set of all the critical points of $f$. The set of \emph{postcritical points} of $f$ is defined as
\[\post(f)=\bigcup_{n\in \N}\{f^n(c)\: c\in {\rm crit}(f)\}.\]We call a map $f$ \emph{postcritically-finite} if the cardinality of $\post(f)$ is finite. Since
\[{\rm crit}(f^n) = {\rm crit}(f) \cup f^{-1}({\rm crit}(f))\cup \cdots \cup f^{-(n-1)}({\rm crit}(f)),\] one can verify that $\post(f)=\post(f^n)$ for any $n\in \N$. So $f$ is postcritically-finite if and only if there is some $n\in \N$ for which $f^n$ is postcritically-finite.

Let $\mathcal{C}\subset \S^2$ be a Jordan curve containing $\post(f)$. We fix a metric $d$ on $\S^2$ that induces the standard metric topology on $\S^2$.
Denote by \emph{${\rm mesh}(f,n,{\mathcal C})$} the supremum of the diameters of all connected components of the set $f^{-n}(\S^2\setminus {\mathcal C})$.

\begin{de} \label{expandingmap}
A branched covering map $f\:\S^2\ra \S^2$ is called a \emph{Thurston map} if $\deg(f)\geq 2$ and $f$ is postcritically-finite. A Thurston map $f\:\S^2\ra \S^2$ is called \emph{expanding} if there exists a Jordan curve $\mathcal{C}\subset \S^2$ with $\mathcal{C} \supset \post(f)$ and
\begin{equation} \label{mesh}
\lim_{n\ra \infty}{\rm mesh}(f,n,{\mathcal C})=0.
\end{equation}
\end{de}

The relation \eqref{mesh} is a topological property, as it is independent of the choice of the metric, as long as the metric induces the standard topology on $\S^2$. Lemma 8.1 in \cite{BMExpanding} shows that if the relation \eqref{mesh} is satisfied for one Jordan curve ${\mathcal C}$ containing $\post(f)$, then it holds for every such curve. One can essentially show that a Thurston map is expanding if and only if all the connected components in the preimage under $f^{-n}$ of any open Jordan region not containing $\post(f)$  become uniformly small as $n$ goes to infinity.

The following theorem (Theorem 1.2 in \cite{BMExpanding}) says that there exists an invariant Jordan curve for some iterate of $f$.
\begin{thm} \label{invariantJordancurve}
If $f \: \S^2\ra \S^2$ is an expanding Thurston map, then for some $n\in \N$ there exists a Jordan curve $\mathcal{C}\subset \S^2$ containing $\post(f)$ such that $\mathcal C$ is invariant under $f^n$, i.e., $f^n({\mathcal C})\subseteq {\mathcal C}$.
\end{thm}
In the following, it is not assumed that the Jordan curve $\mathcal{C}$ is invariant unless stated otherwise.

Recall that an \emph{isotopy} $H$ between two homeomorphisms is a homotopy so that at each time $t\in [0,1]$, the map $H_t$ is a homeomorphism. An \emph{isotopy $H$ relative to a set $A$} is an isotopy satisfying
\[H_t(a)=H_0(a)=H_1(a)\]
for all $a\in A$ and $t\in [0,1]$.

\begin{de} \label{thurstoneq}
Consider two Thurston maps $f\:\S^2\ra \S^2$ and $g\:\S^2_1\ra \S^2_1$, where $\S^2$ and $\S^2_1$ are $2$-spheres. We call the maps $f$ and $g$ \emph{(Thurston) equivalent} if there exist homeomorphisms $h_0,h_1\:\S^2\ra \S^2_1$ that are isotopic relative to $\post(f)$ such that $h_0\circ f=g\circ h_1$.
We call the maps $f$ and $g$ \emph{topologically conjugate} if there exists a homeomorphism $h\:\S^2\ra \S^2_1$ such that $h\circ f=g\circ h$.
\end{de}
For equivalent Thurston maps, we have the following commutative diagram
\[\begin{CD}
\S^2 @>h_1>>\S^2_1 \\
@Vf VV @VVg V\\
\S^2 @>h_0>> \S^2_1  .
\end{CD} \]

If $f\:\S^2\ra\S^2$ is an expanding Thurston map and $g\:\S^2_1\ra \S^2_1$ is topologically conjugate to $f$, then  $g$ is also expanding. If $f$ and $g$ are equivalent Thurston maps and one of them is expanding, then the other one is not necessarily expanding as well. Thus, topological conjugacy is a much stronger condition than Thurston equivalence.
The following theorem (see Theorem 9.2 in \cite{BMExpanding}) shows that under the condition that both maps are expanding, these two relations are the same.
\begin{thm}
Let $f \: \S^2\ra \S^2$ and $g \: \S^2_1 \ra \S^2_1$ be equivalent Thurston
maps that are expanding. Then they are topologically conjugate.
\end{thm}

We now consider the cardinality of the postcritical set of $f$. In Remark 5.5 in \cite{BMExpanding}, it is proved that there are no Thurston maps with $\#\post(f)\leq 1$. Proposition 6.2 in \cite{BMExpanding} shows that all Thurston maps with $\#\post(f)=2$ are Thurston equivalent to a \emph{power map}  on the Riemann sphere,
\[z\mapsto z^k, \mbox{ for some }k\in \Z\setminus\{-1,0,1\}.\]
Corollary 6.3 in \cite{BMExpanding} states that if $f\:\S^2\ra\S^2$ is an expanding Thurston map, then $\#\post(f)\geq 3$.

Let $f\:\S^2\ra\S^2$ be a Thurston map, and let ${\mathcal C}\subset \S^2$ be a Jordan curve containing $\post(f)$. By the Sch\"onflies theorem, the set $\S^2\setminus {\mathcal C}$ has two connected components, which are both homeomorphic to the open unit disk. Let $T_0$ and $T_0'$ denote the closures of these components. They are cells of dimension $2$, which we call \emph{$0$-tiles}. The postcritical points of $f$ are called \emph{$0$-vertices} of $T_0$ and $T_0'$, 
which are cells of dimension $0$. The closed arcs on $\mathcal{C}$ between vertices \emph{$0$-edges} of $T_0$ and $T_0'$, which are cells of dimension $1$. These $0$-vertices, $0$-edges and $0$-tiles form a cell decomposition of $\S^2$, denoted by $\D^0=\D^0(f,{\mathcal C})$. We call the elements in $\D^0$ $0$-cells. Let $\D^1=\D^1(f,{\mathcal C})$ be the set of connected subsets $c\subset \S^2$ such that $f(c)$ is a cell in $\D^0$ and $f|_c$ is a homeomorphism of $c$ onto $f(c)$. Call $c$ a $1$-tile if $f(c)$ is a $0$-tile, call $c$ a $1$-edge if $f(c)$ is a $0$-edge, and call $c$ a $1$-vertex if $f(c)$ is a $1$-vertex. Lemma 5.4 in \cite{BMExpanding} states that $\D^1$ is a cell decomposition of $\S^2$. Continuing in this manner, let $\D^n=\D^n(f,{\mathcal C})$  be the set of all connected subsets of $c\subset \S^2$ such that $f(c)$ is a cell in $\D^{n-1}$ and $f|_c$ is a homeomorphism of $c$ onto $f(c)$, and call these connected subsets $n$-tiles, $n$-edges and $n$-vertices correspondingly, for $n\in\N_0$. By Lemma 5.4 in \cite{BMExpanding}, $\D^n$ is a cell decomposition of $\S^2$, for each $n\in \N_0$, and we call the elements in $\D^n$ $n$-cells. The following lemma lists some properties of these cell decompositions. For more details, we refer the reader to Proposition 6.1 in \cite{BMExpanding}.
\begin{lem}\label{tilenumber}
Let $k,n\in \N_0$, let $f\:\S^2\ra \S^2$ be a Thurston map, let $\mathcal C\subset~ \S^2$ be a Jordan curve with $\mathcal C\supset \post(f)$, and let $m=\#\post(f)$. Consider the associated cell decompositions of $\S^2$ described above.
\begin{enumerate}
  \item 
     If $\tau$ is any $(n+k)$-cell, then $f^k(\tau)$ is an $n$-cell, and $f^k|_{\tau}$ is a homeomorphism of $\tau$ onto $f^k(\tau)$.
  \item Let $\sigma$ be an $n$-cell. Then $f^{-k}(\sigma)$ is equal to the union of all $(n+k)$-cells $\tau$ with $f^k(\tau)=\sigma$.
  \item The number of $n$-vertices is less than or equal to $m\deg(f)^n$, the number of $n$-edges is $m\deg(f)^n$, and the number of $n$-tiles is $2\deg (f)^n$.
  \item The $n$-edges are precisely the closures of the connected components of
       $f^{-n}(\mathcal C)\setminus f^{-n}(\post(f))$. The $n$-tiles are precisely the closures of the connected components of $\S^2\setminus f^{-n}(\mathcal C)$.
  \item Every $n$-tile is an $m$-gon, i.e., the number of $n$-edges and $n$-vertices contained in its boundary is equal to $m$.
\end{enumerate}
\end{lem}

Let $\sigma$ be an $n$-cell. Let $W^n(\sigma)$ be the union of the interiors of all $n$-cells intersecting with $\sigma$, and call $W^n(\sigma)$ the \emph{$n$-flower} of $\sigma$. In general, $W^n(\sigma)$ is not necessarily simply connected. The following lemma (from Lemma 7.2 in \cite{BMExpanding}) says that if $\sigma$ consists of a single $n$-vertex, then $W^n(\sigma)$ is simply connected.
\begin{lem} \label{flower}
Let $f\: \S^2\ra \S^2$ be a Thurston map. Let $\mathcal C$  be a Jordan curve containing $\post(f)$ and consider the corresponding cell decompositions of $\S^2$. If $\sigma$ is an $n$-vertex, then $W^n(\sigma)$ is simply connected. In addition, the closure of $W^n(\sigma)$ is the union of all $n$-tiles containing the vertex $\sigma$.
\end{lem}

\excise{
One of the most important properties of $n$-flowers is that they build a connection between $n$-tiles of different Jordan curves due to the following lemma in \cite[Lemma 7.12]{BMExpanding}.
\begin{lem}\label{flowertile}
Let $\mathcal C$ and $\mathcal C'$ be Jordan curves in $\S^2$ both containing
$\post(f)$. Then there exists a number $M$ such that each $n$-tile for $(f, \mathcal C)$
is covered by $M$ $n$-flowers for $(f, \mathcal C')$.
\end{lem}

\begin{re}
The exact same proof for this lemma shows that for $n'\geq n$, there exists a number $M$ such that each $n'$-tile $(f, \mathcal C)$ is covered by $M$ $n$-flowers for $(f, \mathcal C')$.
\end{re}
}

We obtain a sequence of cell decompositions of $\S^2$ from a Thurston map and a Jordan curve on $\S^2$. In many instances it is desirable that the local degrees of the map $f$ at all the vertices are bounded, and this can be obtained using the assumption that $f$ has no periodic critical points (see  \cite[Lemma 17.1]{BMExpanding}).
\begin{lem} \label{noperiodic}
Let $f : \S^2\ra \S^2$ be a branched covering map. Then f has no
periodic critical points if and only if there exists $N\in \N$ such that
\[\deg_{f^n}(p)\leq N,\]
for all $p\in \S^2$ and all $n\in \N$.
\end{lem}
It can be shown using the lemma above as well as Proposition 12.5 and 13.1 in \cite{BMExpanding} that there exist expanding Thurston maps with periodic critical points.
However, \emph{from now on we will only consider Thurston maps that do not have periodic critical points.}

\begin{de} \label{defofm}
Let $f \: \S^2 \ra \S^2$ be an expanding Thurston map, and let
${\mathcal C}\subset \S^2$ be a Jordan curve containing $\post(f) $. Let $x, y \in \S^2$.
For $x \not= y$ we define
\begin{eqnarray*}
m_{f,\mathcal C}(x, y) := \min\{n\in \N_0 :\mbox{ there exist disjoint $n$-tiles }X \mbox{ and } Y \\
\mbox{ for }(f, \mathcal C)  \mbox{ with } x\in X \mbox{ and } y\in Y \}.
\end{eqnarray*}
If $x = y$, we define $m_{f,\mathcal C}(x, x)= \infty$.
\end{de}

The minimum in the definition above is always obtained since the
diameters of $n$-tiles go to $0$ as $n\ra \infty$. We usually drop one or both
subscripts in $m_{f,\mathcal C}(x, y)$ if $f$ or $\mathcal C$ is clear from the context. If
we define for $x,y\in \S^2$ and $x \not= y$,
\begin{eqnarray*}
m'_{f,\mathcal C}(x, y) = \max\{n\in \N_0 : \mbox{ there exist nondisjoint $n$-tiles }X \mbox{ and } Y  \\
\mbox{ for } (f, \mathcal C) \mbox{ with } x\in X \mbox{ and } y\in Y \},
\end{eqnarray*}
then $m_{f,\mathcal C}$ and $m'_{f,\mathcal C}$ are essentially the same up to a constant (see Lemma 8.6 (v) in \cite{BMExpanding}). Note that our notation for $m$ and $m'$ is switched from that used in \cite{BMExpanding}.

\begin{lem} \label{twom}
Let $m_{f,\mathcal C}$ and $m'_{f,\mathcal C}$ be defined as above. There exists a constant $k>0$, such that for any $x,y\in \S^2$ and $x\not=y$,
\[ m'_{f,\mathcal C}(x,y)-k\leq m_{f,\mathcal C}(x,y)\leq m'_{f,\mathcal C}(x,y)+1.\]
\end{lem}

\begin{de}\label{visual}
Let $f \: \S^2\ra \S^2$ be an expanding Thurston map and let
$d$ be a metric on $\S^2$. The metric $d$ is called a \emph{visual metric} for $f$ if there
exists a Jordan curve $\mathcal C\subset \S^2$ containing $\post(f) $, constants $\Lambda > 1$ and  $C \geq 1$ such that
\[\frac1{C}\Lambda^{-m_{f,\mathcal C}(x, y)} \leq d(x, y) \leq C\Lambda^{-m_{f,\mathcal C}(x, y)}\]
for all $x, y \in \S^2$. The number $\Lambda$ is called the \emph{expansion factor} of the visual metric $d$.
\end{de}

Proposition 8.9 in \cite{BMExpanding} states that for any expanding Thurston map $f\:\S^2\ra \S^2$, there exists a visual metric for $f$ that induces the standard topology on $\S^2$. Lemma 8.10 in the same paper gives the following characterization of visual metrics.
\begin{lem} \label{charvisual}
Let $f \: \S^2 \ra \S^2$ be an expanding Thurston map. Let $\mathcal C\subset \S^2$
be a Jordan curve containing $\post(f)$, and $d$ be a visual metric for $f$ with
expansion factor $\Lambda > 1$. Then there exists a constant $C > 1$ such that
\begin{enumerate}
  \item $d(\sigma,\tau)\geq (1/C)\Lambda^{-n}$ whenever $\sigma$ and $\tau$ are disjoint $n$-cells,
  \item $(1/C)\Lambda^{-n}\leq \diam (\tau)\leq C\Lambda^{-n}$ for $\tau$ any $n$-edge or $n$-tile.
\end{enumerate}
Conversely, if $d$ is a metric on $\S^2$ satisfying conditions $(1)$ and $(2)$
for some constant $C>1$, then $d$ is a visual metric with expansion
factor $\Lambda > 1$.
\end{lem}

Let $(X,d)$ be a metric space. For $\alpha\geq 0$ and for any Borel subset $S\subseteq X$, the \emph{$\alpha$-dimensional Hausdorff measure} $H^{\alpha}(S)$ of $S$ is defined as
\[H^{\alpha}(S):=\lim_{\epsilon\ra 0+}H^{\alpha}_{\epsilon}(S),\]
where
\[H^{\alpha}_{\epsilon}(S)=\inf\left\{ \sum_{i=1}^{\infty} \diam(U_i)^{\alpha} : S\subseteq \bigcup_{i=1}^{\infty}{U_i}\mbox{ and } \diam(U_i)<\epsilon \right\} \]
where the infimum is taken over all countable covers $\{U_i\}$ of $S$. The \emph{Hausdorff dimension} $\dim_H(X)$ of a metric space $X$ is the infimum of the set of $\alpha\in [0,\infty)$ such that $\alpha$-dimensional Hausdorff measure of $X$ is zero:
\[ \dim_H (X) := \inf\left\{\alpha\geq 0 : H^{\alpha}(X)=0\right\} .\]
The \emph{dimension} of a probability measure $\mu$ on $X$ is
\[ \dim \mu := \inf\{\dim_H(E) \: E\subset X \mbox{ is measurable and } \mu(E) = 1\}.\]

The following theorem (\cite[Theorem 4]{MeyDimension}) gives a characterization of Latt\`es maps among among all  expanding rational Thurston maps.
\begin{thm} \label{dim}
Let $f\:\widehat{\C}\ra \widehat{\C}$ be an expanding rational Thurston map. The map $f$ is a Latt\`es map if and only if there exists a visual metric $d$ on $\widehat \C$ such that the dimension of the (normalized standard) Lebesgue measure with respect to the metric $d$ is equal to $2$.
\end{thm}

\section{Latt\`es and Latt\`es-type Maps} \label{lattes}
\noindent
In this section, we introduce Latt\`es-type maps and establish some of their properties. We also briefly review the concept of the orbifold $O_f$ of a Thurston map $f$.

Let $\L,\L'\subset \R^2$ be lattices.
We will always assume that a lattice has rank $2$.  The quotients $\T=\R^2/\L$ and $\T'=\R^2/\L'$ are tori. Let $A\: \R^2\ra \R^2$ be an affine orientation-preserving map such that for any two points $p,q\in \R^2$ with $p-q\in \L$, we have $A(p)-A(q)\in \L'$. The quotient of the map $A$,
\[\bar A\: \T\ra \T',\]
is called an (orientation-preserving) \emph{torus homomorphism}.  If the map $\bar A$ is also bijective, we call the map $\bar A\: \T\ra \T'$ a \emph{torus isomorphism} between $\T$ and $\T'$. If $\L =\L'$, we call the induced map $\bar A\: \T\ra \T$ a \emph{torus endomorphism}. If in addition, the map $\bar A$ is a torus isomorphism, then we call $\bar A$ a \emph{torus automorphism} of $\T$. If $\L =\Z^2$, then
an affine map $A$ that induces a torus endomorphism has the form
\begin{equation} \label{affine}
A\left(\begin{array}{c}
    x \\
    y \\
  \end{array}\right)
=L\left(\begin{array}{c}
    x \\
    y \\
  \end{array}\right)+
  \left(\begin{array}{c}
    x_0 \\
    y_0 \\
  \end{array}\right)\mbox{ for }\left(\begin{array}{c}
    x \\
    y \\
  \end{array}\right)\in \R^2,
  \end{equation}
where $L$ is a $2\times 2$ matrix with integer entries and positive determinant, and
$x_0,y_0\in \Z$. In this case, the map $A$ is a torus automorphism if and only if $L\in \SL(2,\Z)$.

The matrix $L$ is uniquely determined by $\bar A$. Indeed, if affine maps $A$ and $A'$ induce the same torus endomorphism, then $A$ and $A'$ differ by a translation by $\lambda$ according to equation \eqref{affine}, where $\lambda\in \L$. So we can uniquely define the \emph{determinant, trace} and \emph{eigenvalues} of a torus endomorphism $\bar A$ and 
the affine map $A$ to be the determinant, trace and eigenvalues of the matrix $L$ as in equation \eqref{affine}.
Denote\[\det \bar A=\det A= \det L, \quad {\rm tr}(\bar A)={\rm tr}(A)={\rm tr}( L).\]


\begin{de}
We call $\Theta\:\T\ra \S^2$ a \emph{branched cover induced by a rigid action of a group} $G$ on $\T$ if every element of $g\in G$ acts as a torus automorphism and for any $t,t'\in \T$, we have $\Theta(t)=\Theta(t')$ if and only if there exists $g\in G$ such that
\[t=g(t').\]
\end{de}
An equivalent formulation is that $\Theta$ induces a canonical homeomorphism from the quotient space $\T/G$ onto $\S^2$.

\begin{de} \label{lattetype}
Let $\L\subset \R^2$ be a lattice. Let $\bar A$ be a torus endomorphism of $\T=\R^2/\L$ whose eigenvalues have absolute values greater than $1$. Let $\Theta\:\T\ra \S^2$ be a branched covering map induced by a rigid action of a finite cyclic group on $\T$. A map $f\:\S^2\ra \S^2$ is called a \emph{Latt\`es-type map (with respect to a lattice $\L$)} if there exists $\bar A$ as above such that the semi-conjugacy relation $f\circ \Theta =\Theta \circ \bar A$ is satisfied, i.e., the following diagram commutes:
\[\begin{CD}
\T @>\bar{A}>> \T\\
@V\Theta VV @VV\Theta V\\
\S^2 @>f>> \S^2.
\end{CD}.\]
In addition, if a Latt\`es-type map $f$ is rational, then the map $f$ is called a \emph{Latt\`es map}.
\end{de}
We remark that this definition of Latt\`es maps is equivalent to the definition of Latt\`es maps in \cite{MilLattes}.

\begin{eg} \label{latteseg}
Let $A\:\C\ra\C$ be the $\C$-linear map defined by $z\mapsto 2z$, and let $\wp\:\C\ra \widehat\C$ be the Weierstrass elliptic function with respect to the lattice $2\Z^2$. Let $\bar A\: \T\ra \T=\C/2\Z^2$ be induced by $A$, and let $\Theta\:\T\ra\widehat \C$ be induced by $\wp$. Then the map $f$ satisfying the following diagram is well-defined and is a Latt\`es-type map:
\[\begin{CD}
\T @>{\bar A}>> \T\\
@V\Theta VV @VV\Theta  V\\
\widehat \C @>f>> \widehat \C.
\end{CD}.\]
In fact, the map $f$ is a Latt\`es map (see Example \ref{latteseg3}).
We can think of the map $f$ as follows (see the picture below): observe that the unit square $[0,1]^2$ in $\C$ can be conformally mapped to the upper half plane in $\widehat{\C}$; we glue two unit squares $[0,1]^2$ together along their boundaries , and get a pillow-like space which is homeomorphic to $\widehat{\C}$; we color one of the squares black and the other white; we divide each of the squares into 4 smaller squares of half the side length, and color them with black and white in checkerboard fashion; we map one of the small black pillows to the bigger black pillow by Euclidean similarity, and extend the map to the whole pillow-like space by reflection. We refer the reader to Section 1.2 in \cite{BMExpanding} for further discussion of this example.
\vspace{.5cm}
\begin{center}
\mbox{ \scalebox{0.7}{\includegraphics{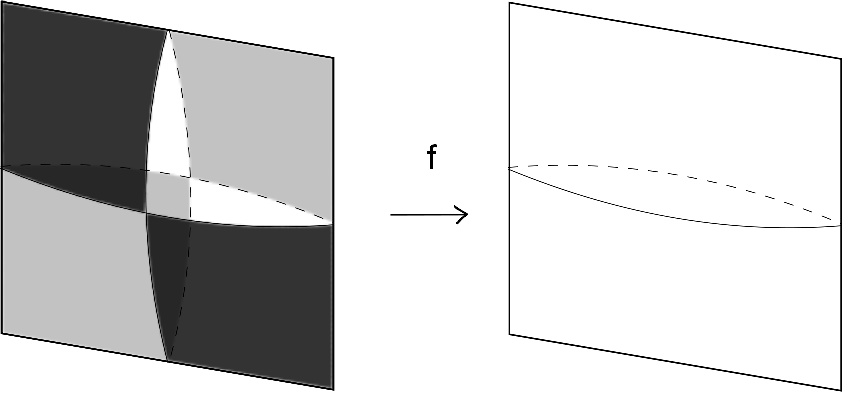}}}
\end{center}
\end{eg}

\vspace{.5cm}

\begin{eg} \label{nonlatteseg}
Let $A\:\R^2\ra\R^2$ be the $\R$-linear map defined by
\begin{equation*}
A\left(\begin{array}{c}
    x \\
    y \\
  \end{array}\right)
=\left(\begin{array}{cc}
    2&0 \\
    0&3 \\
  \end{array}\right)\left(\begin{array}{c}
    x \\
    y \\
  \end{array}\right)
  \mbox{ for }\left(\begin{array}{c}
    x \\
    y \\
  \end{array}\right)\in \R^2,
  \end{equation*}
and let $\wp$ be the Weierstrass elliptic function with respect to the lattice $2\Z^2$.
Let $\bar A\: \T\ra \T=\R^2/2\Z^2$ be induced by $A$, and $\Theta\:\T\ra\S^2$ be induced by $\wp$. Then the map $g$ satisfying the following diagram is well-defined and is a Latt\`es-type map:
\[\begin{CD}
\T @>{\bar A}>> \T\\
@V\Theta VV @VV\Theta  V\\
\S^2 @>g>> \S^2.
\end{CD}.\]
In fact, the map $g$ is not topologically conjugate to a Latt\`es map (see Example \ref{latteseg3}). See the picture below:
\vspace{0.5cm}

\begin{center}
\mbox{ \scalebox{0.7}{\includegraphics{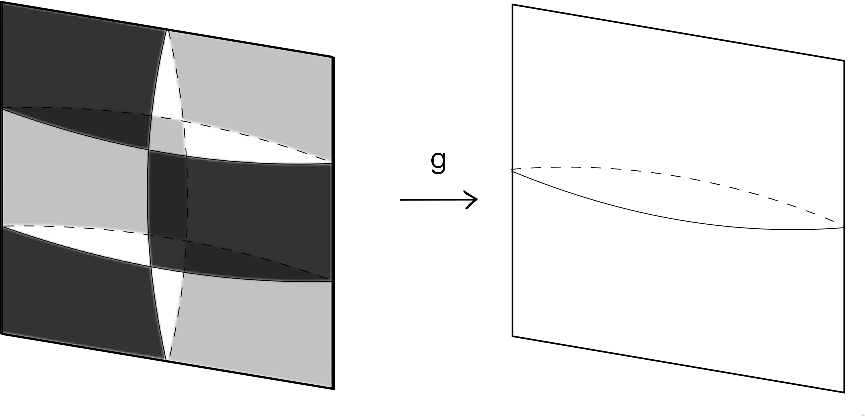}}}
\end{center}

\bigskip
We refer the reader to Example 12.13 in \cite{BMExpanding} for further discussion of this example.
\end{eg}

\begin{lem} \label{integerlattice}
A Latt\`es-type map $f$ over any lattice $\L$ is a Latt\`es-type map over the integer lattice $\Z^2$.
\end{lem}

\begin{proof}
For any Latt\`es-type map $f$ over a lattice $\L$, let $\T=\R^2/\L$. There exist a torus endomorphism
\[\bar A\:\T\ra \T\]
and a branched covering map $\Theta\: \T\ra \S^2$ induced by a rigid action of a fixed cyclic group on $\T$,
such that $f\circ \Theta =\Theta \circ \bar A$. Let $\T_0=\R^2/\Z^2$.  Since $\L$ is a lattice with rank 2, there is an orientation-preserving isomorphism $L\: \Z^2 \ra \L$. This isomorphism $L$ can be extended to an $\R$-linear map of $\R^2$, still denoted by $L$, which induces a torus isomorphism $\bar L\:\T_0\ra \T$. Define a map
\[\bar A_0\:\T_0\ra \T_0\]
by $\bar A_0=\bar L^{-1}\circ \bar A\circ \bar L$, and a branched covering map $\Theta_0\: \T_0\ra \S^2$ by $\Theta_0=\Theta\circ \bar L$. Then $\bar A_0$ is a torus endomorphism and the branched covering map $\Theta_0$ is induced by a rigid action of a finite cyclic group on $\T_0$. In addition,
\[f\circ \Theta_0=f\circ\Theta\circ \bar L=\Theta \circ \bar A\circ \bar L= (\Theta\circ \bar L)\circ (\bar L^{-1}\circ \bar A\circ \bar L)=\Theta_0\circ \bar A_0\](see the diagram below). It follows that the map $f$ is a Latt\`es-type map over the lattice $\Z^2$.
\[\xymatrix@!0{
  & &\T_0 \ar[dll]_{\bar L} \ar[rr]^{\bar A_0} \ar[dddll]^{\Theta_0}
      &  & \T_0 \ar[dddll]^{\Theta_0} \ar[dll]_{\bar L}       \\
  \T \ar[rr]^{\bar A}\ar[dd]_{\Theta}
    &  &   \T \ar[dd]^{\Theta}& \\
  & & &  &                \\
  \S^2 \ar[rr]^{f} & &  \S^2  &     }\]
\end{proof}

\begin{re}
Notice that the proof of this lemma works for   
any rank-2 lattice besides the integer lattice $\Z^2$. Hence, we can choose the lattice for our convenience.
\end{re}

\begin{lem} \label{group}
If a branched covering map $\Theta\:\T\ra \S^2$ is induced by a rigid action of a finite cyclic group $G$ on $\T$, then $G$ acts on $\T$ by rotation   around a fixed point with order of $G$ either $2,3,4$ or $6$.
\end{lem}

Here by $G$ acting on $\T$ by rotation around a fixed point, we mean that if we identify the fixed point with the origin in $\R^2$, and $\T$ with the fundamental domain in $\R^2$, then $G$ acts as a rotation on the Euclidean space $\R^2$.

\begin{proof}
By Lemma \ref{integerlattice}, we may assume that $\T=\R^2/\Z^2$.
Let $g$ be a generator of $G$ with order $n$. The order $n$ is greater than 1 since
\[\T/G=\S^2.\] The element $g$ is an automorphism of the torus $\T$; so $g$ is induced by an affine map $A_g$ on $\R^2$ of the form, 
\[A_g\left(\begin{array}{c}
    x \\
    y \\
  \end{array}\right)
=L_g\left(\begin{array}{c}
    x \\
    y \\
  \end{array}\right)+
  \left(\begin{array}{c}
    x_g \\
    y_g \\
  \end{array}\right) \mbox{ for }\left(\begin{array}{c}
    x \\
    y \\
  \end{array}\right)\in \R^2,\]
where $L_g\in \SL(2,\Z)$ and $x_g,y_g\in \R $.
Since $L_g\in \SL(2,\Z)$ and
\[L_g^n=I_2,\] where $I_2\in \SL(2,\Z)$ is the identity element,
the matrix $L_g$ is conjugate to a rotation of $\R^2$,
\[\left(\begin{array}{cc}\cos \frac{2\pi}{n}& \sin\frac{2\pi}{n}  \\-\sin\frac{2\pi}{n}& \cos\frac{2\pi}{n} \\\end{array}\right).\]
In addition, since the trace of $L_g$ is an integer, we must have
\[2\cos\frac{2\pi}{n}\in \Z.\] Hence, the order $n$ of the group $G$ can only be $2,3,4$ or $6$.

Since $A_g^n={\rm id}_{\R^2}$,
\begin{eqnarray} \label{Agid}
(L_g^n+L_g^{n-1}+\ldots+ L_g+I_2)
\left(\begin{array}{c}
    x_g \\
    y_g \\
\end{array}\right)
=\left(\begin{array}{c}
    a \\
    b \\
  \end{array}\right),
\end{eqnarray}
where $a,b\in \Z$.
Since $L_g$ is conjugate to a non-trivial rotation, $(L_g-I_2)$ is invertible. Multiplying equation \eqref{Agid} by $(L_g-I_2)$, we have
\begin{eqnarray*}
(L_g-I_2)(L_g^n+L_g^{n-1}+\ldots+ L_g+I_2)
\left(\begin{array}{c}
    x_g \\
    y_g \\
\end{array}\right)
=(L_g-I_2)\left(\begin{array}{c}
    a \\
    b \\
  \end{array}\right),
\end{eqnarray*}
so
\begin{eqnarray*}
(L_g-I_2)
\left(\begin{array}{c}
    x_g \\
    y_g \\
\end{array}\right)
=(L_g^{n+1}-I_2)
\left(\begin{array}{c}
    x_g \\
    y_g \\
\end{array}\right)
=(L_g-I_2)\left(\begin{array}{c}
    a \\
    b \\
  \end{array}\right).
\end{eqnarray*}
Hence, we have \begin{eqnarray*}
\left(\begin{array}{c}
    x_g \\
    y_g \\
\end{array}\right)
=\left(\begin{array}{c}
    a \\
    b \\
  \end{array}\right)\in \Z^2,
\end{eqnarray*}
and there exists a fixed point on $\T=\R^2/\Z^2$ under $g$.
Therefore, the group $G$ acts on $\T$ by rotation around a fixed point with order of $G$ either $2,3,4$ or $6$.
\end{proof}

\begin{lem} \label{lattesexpanding1}
Every Latt\`es-type map $f$ is a Thurston map.
\end{lem}

\begin{proof}
By Lemma \ref{integerlattice}, we know that $f$ is a Latt\`es-type map over the lattice $\Z^2$. Let $\T=\R^2/\Z^2$.  There exists a torus endomorphism
\[\bar A\:\T\ra \T\]
and  a branched covering map $\Theta\:\T\ra \S^2$ induced by a rigid action of a finite cyclic group $G$ on $\T$, such that $f\circ\Theta=\Theta \circ \bar A$. Let $A\:\R^2\ra \R^2$ be an affine map inducing $\bar A$.

The map $\Theta \circ \bar A$ is a branched covering map since locally $\bar A$ is a homeomorphism and $\Theta$ is a branched covering map.
Since $f\circ \Theta= \Theta \circ \bar A$ and $\bar A$ has local degree $1$ on every point,
we have
\[\deg_{f}(\Theta(z))\deg_{\Theta}(z)=\deg_{\Theta}(\bar A(z)),\]
and $f$ can be locally written as $w\mapsto w^{d}$, where $w=\Theta(z)$ and
\[d=\deg_{f}(w)=\deg_{\Theta}(\bar A(z))/\deg_{\Theta}(z).\]
Hence, a Latt\`es-type map $f$ is a branched covering map.


Let $V_f$ and $V_{\Theta}$ be the sets of critical values of $f$ and $\Theta$, respectively. We claim that $\post(f)= V_{\Theta}$ and these sets have finite cardinality. The claim is proved similarly to Lemma 3.4 in \cite{MilLattes}. We give the details of the argument for the convenience of the reader. Since $\bar A$ is a local homeomorphism and $\bar A$ and $\Theta$ are both surjective, a point $p\in \S^2$ is a critical value of $\Theta$ if and only if either $p$ is a critical value of $f$, or $p$ has a preimage in $f^{-1}(p)$ that is a critical value of $\Theta$. So $V_{\Theta}=V_f\cup f(V_{\Theta})$. Hence $f(V_f)\subseteq f(V_{\Theta})\subseteq V_{\Theta}$ and inductively, we have $\post(f)\subseteq V_{\Theta}$. The set $V_{\Theta}$ is finite due to the compactness of $\S^2$, and hence the set of critical points of $\Theta$ is also finite.

In order to show that $V_{\Theta}$ is a subset of $\post(f)$, we argue by contradiction. Then there exists a critical point $t_0\in \T$ of $\Theta$ such that $\Theta(t_0)\not\in\post(f)$. There exists $t_1\not=t_0$ in the preimage of $t_0$ under $\bar A$, and there exists $t_2\not=t_1,t$ in the preimage of $t_1$ under $\bar A$. Continuing in this manner, we get a sequence $\{t_i\}$ and the cardinality of $\{t_i\}$ is not finite. For all $i\geq1$, we have
\begin{equation}\label{degreeformula}
    \deg_{f}(\Theta(t_i))\deg_{\Theta}(t_i)=\deg_{\Theta}(\bar A(t_i)) =\deg_{\Theta}(t_{i-1})
\end{equation}
On the other hand, since $\Theta(t_0)\not\in\post(f)$, every element of $f^{-i}(\Theta(t_0))$ is a non-critical point for $f$ and has degree $1$. In particular, $\deg_{f}(\Theta(t_i))=1$ for all $i\geq 0$, and equation \eqref{degreeformula} implies that
\[\deg_{\Theta}(t_i) =\deg_{\Theta}(t_{i-1})=\ldots=\deg_{\Theta}(t_{0})>1.\]
Hence, each $t_i$ is a critical point of $\Theta$. This is a contradiction to the finiteness of the critical set of $\Theta$.

We claim that $\deg(f)=\det(A)$. Since $f\circ \Theta=\Theta\circ \bar A$ and $\deg(\Theta)<\infty$, we have \[\deg(f)=\deg(\bar A).\]
The map $\bar A$ carries a small region of area $\epsilon$ to a region of area $\det(\bar A)\epsilon$, so
\[\deg(\bar A)=\det (\bar A).\]
The claim follows. Since $\deg(f)=\det(A)>1$, the Latt\`es map $f$ is a Thurston map.
\end{proof}



For $a,b\in \N\cup\{\infty\}$, we use the convention that $\infty$ is a multiple of any positive integer or itself. If $a$ is a multiple of $b$, we write $b\,|\,a$. We also use the notation gcd$\{a,b\}$ as the greatest common divisor for $a,b\in \N\cup\{\infty\}$ (defined in the obvious way). Recall that in a set $X$ with partial order $\leq$, an element $x\in X$ is called a \emph{minimal element} if for all $y\in X$ we have that $y\leq x$ implies that $y=x$; an element $x\in X$ is called the \emph{minimum} if for all $y\in X$ we have that $x\leq y$. It is easy to see that if the minimum exists, then it is unique.

\begin{lem} \label{smallfuntion}
For any Thurston map $f$, there exists a function $\nu_f$ that is the minimum among functions $\nu\: \S^2\ra \N\cup \{\infty\}$ such that
\begin{equation}\label{divisor}
    \nu(p)\deg_f(p)\big|\nu(f(p))
\end{equation}
for all $p\in \S^2$.
\end{lem}

\begin{proof}
We have a natural partial order for functions satisfying \eqref{divisor}. If $\nu_1$ and $\nu_2$ are such functions, then we set 
\[\nu_1\leq \nu_2\mbox{ iff }\nu_1(p)| \nu_2(p)\] for all $p\in \S^2$.
In order to show the existence of such a minimal function satisfying \eqref{divisor}, we set
$\nu(p)=1$ if $p$ is not a postcritical point of $f$. We only need to assign a value to the finitely many postcritical points of $f$. If we let $\nu(p)=\infty$ when $p\in \post(f)$, this shows the existence of a such function $\nu$.
The existence of a minimal function follows from assigning values over a fixed finite set.

To show uniqueness of a minimal function, suppose that $\nu_1$ and $\nu_2$ are both minimal functions satisfying condition \eqref{divisor}. Let
\[\nu_3(p):={\rm gcd}\{\nu_1(p),\nu_2(p)\}.\]
We claim that $\nu_3$ satisfies condition \eqref{divisor}. Indeed,
\[{\rm gcd}\{\nu_1(f(p)),\nu_2(f(p))\}=\nu_3(f(p))\] is a multiple of
\begin{eqnarray*}
  {\rm gcd}\{\nu_1(p)\deg_f(p),\nu_2(p)\deg_f(p)\} &=& {\rm gcd}\{\nu_1(p),\nu_2(p)\}\deg_f(p) \\
   &=& \nu_3(p)\deg_f(p).
\end{eqnarray*}
Hence, we have $\nu_3\leq \nu_1,\nu_2$. Since $\nu_1$ and $\nu_2$ are both minimal functions, we conclude that
\[\nu_1=\nu_2=\nu_3.\]

We claim that this unique minimal function $\nu_f$ is the minimum with respect to the order $\leq$. Indeed, let $\nu$ be a function satisfying \eqref{divisor}. Then we have that
\[\nu_1={\rm gcd}(\nu_f,\nu)\leq \nu_f\] also satisfies \eqref{divisor}. Hence, $\nu_1=\nu_f$ by the minimality of $\nu_f$. Therefore, $\nu_f={\rm gcd}(\nu_f,\nu)\leq \nu$.
\end{proof}

Thurston associated an \emph{orbifold} $O_f=(\S^2,\nu_f)$ to a Thurston map $f$ through the smallest $\nu_f$ function in Lemma \ref{smallfuntion} (see \cite{DHThurston}). More precisely, for each $p\in \S^2$ with $\nu_f(p)\not=1$, the point $p$ is a \emph{cone point} with \emph{cone angle} $2\pi/v_f(p)$. For $\post(f)=\{p_1,\ldots,p_m\}$, use $(\nu(p_1),\ldots,\nu(p_m))$ to denote the \emph{type} of $O_f$. We will not elaborate on the geometric significance of the orbifold here, but instead refer the reader to Chapter 13 in \cite{ThuGeometry}.

\begin{de}
For any Thurston map $f$ and the smallest function $\nu_f\:\S^2\ra \N\cup\{\infty\}$ associated to $f$  satisfying condition \eqref{divisor}, let
\[\chi(O_f)=2-\sum_{p\in \post(f)}\left(1-\frac{1}{\nu_f(p)}\right).\]
\begin{itemize}
  \item If $\chi(O_f)=0$, we say that the orbifold $O_f$ is \emph{parabolic};
  \item If $\chi(O_f)<0$, we say that the orbifold $O_f$ is \emph{hyperbolic}.
\end{itemize}
We call $\chi(O_f)$ the \emph{Euler characteristic} of the orbifold $O_f$ associated to $f$.
\end{de}

\begin{re}
By Proposition 9.1 (i) in \cite{DHThurston}, $\chi(O_f)\leq 0$.
\end{re}

\excise{
\begin{de}
For any Thurston map $f$ and the smallest function $v_f\:\S^2\ra \N\cup\{\infty\}$ associated to $f$ satisfying condition \eqref{divisor}, let
\[\chi(f)=2-\sum_{p\in \post(f)}\left(1-\frac{1}{\nu(p)}\right).\]
\begin{itemize}
  \item If $\chi(f)>0$, we say that the map $f$ has elliptic type;
  \item If $\chi(f)=0$, we say that the map $f$ has parabolic type;
  \item If $\chi(f)<0$, we say that the map $f$ has hyperbolic type.
\end{itemize}
\end{de}
}

\begin{lem} \label{noofpost}
For a Latt\`es-type map $f$, the orbifold $O_f$ is parabolic. In particular, the number of cone points must be either three or four. Hence, the cardinality of the postcritical set of $f$ is either three or four.
\end{lem}

\begin{proof}
There exist a torus endomorphism $\bar A\:\T\ra \T$ and a branched covering map $\Theta\: \T\ra \S^2$ induced by a group action on $\T$ as a rotation around some base point in $\T$, such that $f\circ \Theta =\Theta \circ \bar A$. For any points $t_1\in \bar A^{-1}(t_0)$, $t_i\in \T$, and $p_1\in f^{-1}(p_0)$, $p_i\in \S^2$ such that $\Theta(t_i)=p_i$, $i=0,1$, we have that
\[\deg_{\Theta}(t_0)=\deg_f(p_1) \deg_{\Theta}(t_1). \]
Define $\nu(\Theta(t))=\deg_{\Theta}(t)$, and $\nu(p)=1$ if $p\not\in\post(f)$. Since $\Theta$ is induced by a group action, different preimages of $\Theta(t)$ under $\Theta$ all have the same degree. Explicitly, for any $t, t'\in \T$ such that $\Theta(t)=\Theta(t')$, there is a torus automorphism $g$ such that $g(t)=t'$ and $\Theta(g(x))=\Theta(x)$ for all $x\in \T$, so
\[\deg_{\Theta}(t)=\deg_{\Theta}(g(t))\deg_g(t)=\deg_{\Theta}(t').\]
In the proof of Lemma \ref{lattesexpanding1}, we showed that $\post(f)$ is equal to the set of critical values of $\Theta$, so $\nu$ is well-defined on $\S^2$. In addition,
\[\nu(p_0)=\nu(f(p_1))=\deg_{\Theta}(t_0)=\deg_f(p_1) \deg_{\Theta}(t_1)=\deg_f(p_1) \nu(p_1).\] So $\nu$ is a function satisfying condition \eqref{divisor}.

We claim that $\nu$ is the smallest function satisfying condition \eqref{divisor}. Indeed, suppose that $\nu'$ satisfying condition \eqref{divisor} is smaller than $\nu$. If $p_1\not\in \post(f)$, then
\[\nu'(p_0)=\nu'(f(p_1))=\deg_f(p_1) =\nu(f(p_1))=\nu(p_0).\]
If $p_1\in \post(f)$, then there exists $n>0$ and $p\in f^{-n}(p_1)$, such that $p\not\in \post(f)$. By induction on $n$, we get that $\nu'(p_1)=\nu(p_1)$ and hence
\[\nu'(p_0)=\nu'(f(p_1))=\deg_f(p_1) \nu'(p_1)=\deg_f(p_1) \nu(p_1)=\nu(f(p_1))=\nu(p_0).\]
Thus, $\nu'=\nu$ and our claim is proved.

By the proof of Proposition 9.1 (i) in \cite{DHThurston}, $\nu(f(p_1))=\deg_f(p_1) \nu(p_1)$ implies that $f$ is a covering map of orbifolds $f\: O_f \ra O_f$, and again by Proposition 9.1 (ii), $\chi(O_f)=0$ and $O_f$ is parabolic.  All the parabolic orbifolds are classified in Section 9 in \cite{DHThurston}, and they have type $(2,2,2,2), (3,3,3), (2,3,6)$ and $(2,4,4)$, and all have three or four cone points.
\end{proof}

\begin{pro}\label{lattesexpanding}
Every Latt\`es-type map $f$ is an expanding Thurston map.
\end{pro}

\begin{proof}
By Lemma \ref{lattesexpanding1}, we know that $f$ is a Thurston map.
Given a Latt\'es-type map $f$, there exists a torus endomorphism
\[\bar A\:\T\ra \T\]
and  a branched covering map $\Theta\:\T\ra \S^2$ induced by a rigid action of a finite cyclic group $G$ on $\T$, such that $f\circ\Theta=\Theta \circ \bar A$.

Let $\mathcal C\subset \S^2$ be a Jordan curve containing $\post(f)$.
The torus $\T$ carries a flat metric induced by the Euclidean metric $\R^2$, and the map $\Theta$ induces a flat orbifold metric on $\T/G\cong \S^2$.
Observe that the interior $T$ of a $0$-tile on $\S^2$ under the cell decomposition of $(f,{\mathcal C})$ does not intersect with $\post(f)=V_{\Theta}$, so $\Theta$ restricted to one of the connected components $T'$ of $\Theta^{-1}(T)$ is a homeomorphism. In addition, since $\S^2$ is obtained by a finite quotient of $\T$ by $G$,
\[\diam(T')\leq \diam(T)\leq 2|G|\diam(T').\]
Each connected component of $\bar A^{-1}(T')$ has diameter $\lambda\diam(T')$, where
\[\lambda=|\lambda_1|^{-1}<1\]
is the inverse of the smaller absolute value of the eigenvalues of $\bar A$. Hence, each connected component of the preimage of $T$ under $f^n$ has diameter bounded by $2|G|\lambda^n\diam(T)$, where $\lambda<1$. Therefore, we have
\[{\rm mesh}(f,n,{\mathcal C})\leq 2|G| \lambda^n {\rm mesh}(f,1,{\mathcal C}) \ra 0\]
as $n\ra \infty$, and the map $f$ is expanding.
\end{proof}

\section{Combinatorial Expansion Factor and $D_n$} \label{necessity}
\noindent
In this section, we first review the definitions and some properties of the quantity $D_n$ and the related combinatorial expansion factor of an expanding Thurston map. Then we prove a relation between $D_n$ and the operator norm of the associated torus map for Latt\`es-type maps, which gives the necessity of the third condition in Theorem \ref{main}.

Let $f\:\S^2\ra \S^2$ be an expanding Thurston map and let $\mathcal{C}$ be a Jordan curve containing $\post(f)$. First, we review some definitions and propositions from \cite{BMExpanding}.
\begin{de}\label{joinoppositesides}
A set $K\subseteq \S^2$ \emph{joins opposite sides} of $\mathcal C$ if $\#$$\post(f)\geq 4$
and $K$ meets two disjoint $0$-edges, or if $\#$$\post(f) = 3$ and $K$ meets all
three $0$-edges.
\end{de}

Let $D_n=D_n(f,\mathcal C)$ be the minimum number of $n$-tiles needed to join opposite sides of a Jordan curve $\mathcal{C}$. More precisely,
\begin{eqnarray}\label{defdn}
\hspace{1cm}
D_n :=\min\{N\in \N: \mbox{ there exist $n$-tiles } X_1, . . . ,X_N \mbox{ such that }  \\
\bigcup_{j=1}^N X_j\mbox{ is connected and joins opposite sides of }\mathcal C\}. \nonumber
\end{eqnarray}
We will often abuse notation and write $D_n$ rather than $D_n(f,\mathcal C)$.

\begin{eg} \label{latteseg2}
Recall the Latt\`es-type map  $f$ in Example \ref{latteseg} which is induced by the map $z\mapsto 2z$. The postcritical set $\post (f)$ consists of the four common corner points of the two big squares. If we let $\mathcal C$ be the common boundary of the two big squares, then $\mathcal C$ contains $\post (f)$ and
\[D_n(f,\mathcal C)=2^n\]
for all $n\geq 0$.
Similarly, consider the Latt\`es-type map $g$ in Example \ref{nonlatteseg}. If we let $\mathcal C'$ be the boundary of the common big squares, then $\mathcal C'$ contains $\post (g)$ which consists of the four corner points, and
\[D_n(g,\mathcal C')=2^n\]
for all $n\geq 0$.
\end{eg}


We are going to need Lemma 7.9 in \cite{BMExpanding} in Section \ref{sufficiency}. It states that:
\begin{lem} \label{disjointcell}
Let $n\in \N_0$, and let $K\subset \S^2$ be a connected set. If there exist two disjoint $n$-cells $\sigma$ and $\tau$ with $K\cap \sigma\not=\emptyset$ and $K\cap \tau\not=\emptyset$, then $f^n(K)$ joins opposite sides of $\mathcal C$.
\end{lem}

Lemma 7.10 in \cite{BMExpanding} states that:
\begin{lem} \label{disjointcell1}
For $n,k\in \N_0$, every set of $(n+k)$-tiles whose union is
connected and meets two disjoint $n$-cells contains at least $D_k$ elements.
\end{lem}

Proposition 17.1 in \cite{BMExpanding} says that:
\begin{pro} \label{expansionfactor}
For an expanding Thurston map $f\: \S^2\ra \S^2$, and a Jordan curve $\mathcal C$ containing $\post(f)$,
the limit
\[\Lambda_0(f):=\lim_{n\ra\infty}D_n(f,\mathcal{C})^{1/n}\] exists and is independent of $\mathcal C$.
\end{pro}

We call $\Lambda_0(f)$ the \emph{combinatorial expansion factor} of $f$.

Proposition 17.2 in \cite{BMExpanding} states that:
\begin{pro}
If $f\:\S^2\ra \S^2$ and $g\:\S^2_1\ra \S^2_1$ are expanding Thurston maps that are topologically conjugate, then $\Lambda_0(f)=\Lambda_0(g)$.
\end{pro}

Let $f$ be an expanding Thurston map. For any two Jordan curves $\mathcal{C}$ and $\mathcal{C'}$ with $\post(f)\subset \mathcal C,\mathcal C'$, inequality (17.1) in \cite{BMExpanding} states that there exists a constant $c>0$ such that for all $n>0$,
\[ \frac1{c}D_n(f,\mathcal C)\leq D_n(f,\mathcal C')\leq c D_n(f,\mathcal C).\]
We obtain the following lemma:

\excise{
\begin{eqnarray*}
\lim_{n\ra\infty}D_n(f,\mathcal{C'})^{1/n} &=& \lim_{n\ra\infty}D_n(f,\mathcal{C})^{1/n}, \\
{\rm i.e. }\quad  \lim_{n\ra\infty}\left(\frac{D_n(f,\mathcal{C'})}{D_n(f,\mathcal{C})}\right)^{1/n} &=& 1, \\
\end{eqnarray*}
so there exists $\delta\leq 1$, such that
\[\left(\frac{D_n(f,\mathcal{C'})}{D_n(f,\mathcal{C})} \right)^{1/n}\geq  \delta. \]
Hence, for any Jordan curves $\mathcal{C}$, $\mathcal{C'}$ containing $\post(f)$,
\begin{equation*}
D_n(f,\mathcal{C'}) \geq \delta^n D_n(f,\mathcal{C}),
\end{equation*}
where $\delta\leq 1$.
}

\begin{lem} \label{dn}
With the notation above, there exists a constant $c>0$ such that $D_n(f,\mathcal C)\geq c(\deg f)^{n/2}$ if and only if there exists a constant $c'>0$ such that $D_n(f,\mathcal C')\geq c'(\deg f)^{n/2}$ for all $n>0$.
\end{lem}
So we may say that $D_n\geq c(\deg f)^{n/2}$ for some $c>0$ without specifying Jordan curves.

\begin{lem}
Let $f$ and $g$ be two expanding Thurston maps that are topologically conjugate via a homeomorphism $h$. Let $\mathcal C$ be a Jordan curve on $\S^2$ containing $\post(f)$, and let $\mathcal C'$ be the image of $\mathcal C$ under $h$. Then \[D_n(f,\mathcal C)=D_n(g,\mathcal C')\] for all $n\geq 0$.
\end{lem}

This lemma follows directly from the definitions of $D_n$ and topological conjugacy.

\bigskip

Recall that the \emph{maximum norm} (or \emph{$l^{\infty}$ norm}) of a vector
\[v=(x_1,\ldots,x_n)\in~\R^n\] is
\[\|v\|_{\infty}=\max\{|x_1|,\ldots,|x_n|\}.\]
Let $A\:\R^n\ra \R^n$ be an $\R$-linear map. Then the \emph{($l^{\infty}$-)operator norm} is
\[\|A\|_{\infty}:= \max\{ \left\|Av\right\|_{\infty}\:v\in \R^n,\|v\|_{\infty}=1\}.\]
\smallskip

Let $f$ be a Latt\`es-type map over a lattice $\L$ with orbifold type $(2,2,2,2)$. There exist a torus endomorphism $ \bar A\:\T\ra \T$ and a branched covering map $\Theta\: \T\ra \S^2$ induced by a group action on $\T$ such that $f\circ \Theta =\Theta \circ  \bar A$, where $\T=\R^2/\L$. We use $A$ to denote an affine map lifted to the covering of $\T$ with $L$ as the corresponding linear map. By the remark after Lemma \ref{integerlattice}, we may assume $\L=2\Z^2$. For a Latt\`es-type map with orbifold type $(2,2,2,2)$, we can identify $\Theta\circ p\:\R^2\ra\S^2$, where $p\: \R^2\ra \R^2/(2\Z^2)$ is the quotient map, with the Weierstrass function $\wp\: \R^2\ra \S^2$ with the lattice $2\Z^2$.
See the diagram below.
\[\xymatrix{
  \R^2 \ar[d]_{\wp} \ar[r]^{A}
                & \R^2 \ar[d]^{\wp}  \\
  \S^2  \ar[r]_{f}
                & \S^2 .            }
\]
Let the Jordan curve $\mathcal{C}$ on $\S^2$ be the image of the boundary of the unit square $[0,1]\times[0,1]$ under $\wp$.

\begin{pro}\label{dnrelation}
Let $f$ be a Latt\`es-type map with orbifold type $(2,2,2,2)$. Let $A$ be its affine map from $\R^2$ to $\R^2$ with $L$ as the corresponding linear map and $\wp\: \R^2\ra \S^2$ be the Weierstrass function with the lattice $2\Z^2$(as in the remark above). We have
\[\frac{1}{\|L^{-n}\|_{\infty}}\leq D_n(f,\mathcal{C})\leq \frac{1}{\|L^{-n}\|_{\infty}}+1,\]
where the Jordan curve $\mathcal{C}$ is the image of the boundary of the unit square $[0,1]\times[0,1]$ under $\wp$.
\end{pro}

\begin{proof}
\excise{
If $f$ is any Latt\`es-type map over a lattice $\L$, there exist a torus endomorphism $ \bar A\:\T\ra \T$ and a branched covering map $\Theta\: \T\ra \S^2$ induced by a group action on $\T$, such that $f\circ \Theta =\Theta \circ  \bar A$, where $\T=\R^2/\L$. We use $A$ to denote an affine map lifted to the covering of $\T$. By Lemma \ref{integerlattice}, we may assume $\L=2\Z^2$. For a Latt\`es-type map with orbifold type $(2,2,2,2)$, we can identify $\Theta\circ p\:\R^2\ra\S^2$, where $p\: \R^2\ra \R^2/(2\Z^2)$ is the quotient map, with the Weierstrass function $\wp\: \R^2\ra \S^2$ with the lattice $2\Z^2$.
See the diagram below.
\[\xymatrix{
  \R^2 \ar[d]_{\wp} \ar[r]^{A}
                & \R^2 \ar[d]^{\wp}  \\
  \S^2  \ar[r]_{f}
                & \S^2 .            }
\]
Let the Jordan curve $\mathcal{C}$ on $\S^2$ be the image of the boundary of the unit square $[0,1]\times[0,1]$ under $\wp$. }

The idea of the proof is to lift everything to $\R^2$. Since the unit square is homeomorphic to a $0$-tile, $D_n(f,\mathcal{C})$ is the same as the number of pre-images of the unit squares under $A^n$ needed to join the opposite sides of the unit square. We present the details below.

Notice that the pre-image of $\mathcal{C}$ under $\wp$ is the whole grid of $\Z^2$ (i.e., the union of all the horizontal and vertical lines containing an integer-valued point), and $\mathcal{C}$ contains all the post-critical points of $f$. The restriction of $\wp$ to the interior of the rectangle $R_0:=[0,2]\times[0,1]$ is a homeomorphism onto its image, which is the union of the interiors of the $0$-tiles of $\S^2$ and one edge of a $0$-tile. Notice that the same holds for any rectangle obtained from two adjacent unit squares. The pre-images of unit squares under $A^n$ are parallelograms, which we call \emph{$n$-parallelograms}.

The $n$-tiles of $(f,\mathcal{C})$ (i.e., the pre-images of $0$-tiles under $f^n$) are the images of $n$-parallelograms under $\wp$. Let $D_v$ be the minimum number of $n$-parallelograms connecting the line $\{0\}\times(-\infty,+\infty)$ and $\{1\}\times(-\infty,+\infty)$, and let $D_h$ be the minimum number of $n$-parallelograms connecting the lines $(-\infty,+\infty)\times\{0\}$ and $(-\infty,+\infty)\times\{1\}$. We define \[D'_n:=\min\{D_v,D_h\}.\]

We claim that $D_n=D'_n$.
Let $T_1,T_2,\ldots, T_{D_n}$ be a sequence of $n$-tiles with the minimum number of $n$-tiles joining opposite sides of a 0-tile. Without loss of generality, we may assume that this 0-tile is the image of $[0,1]\times[0,1]$ under $\wp$, and the opposite sides of the 0-tile are the images of the sides $[0,1]\times\{0\}$ and $[0,1]\times\{1\}$. Let $T_1'$ be the connected component of $\wp^{-1}(T_1)$ intersecting with $[0,1]\times[0,1]$, which is an $n$-parallelogram. Let $T_2'$ be the component of $\wp^{-1}(T_2)$ intersecting with $T_1'$, which is also an $n$-parallelogram. Let $T_3'$ be the component of $\wp^{-1}(T_3)$ intersecting with $T_2'$, and so on. We obtain a sequence of $n$-parallelograms $T_1',\ldots, T_{D_n}'$ connecting $(-\infty,+\infty)\times\{0\}$ and $(-\infty,+\infty)\times\{1\}$, and hence $D_n'\leq D_n$. 
On the other hand, suppose that a sequence of $n$-parallelograms $P_1,\ldots, P_m$ connects $(-\infty,+\infty)\times\{0\}$ and $(-\infty,+\infty)\times\{1\}$, or connects $\{0\}\times(-\infty,+\infty)$ and $\{1\}\times(-\infty,+\infty)$. Then the sequence $\wp(P_1),\ldots, \wp(P_m)$ of $n$-tiles connects a pair of opposite sides of a 0-tile. We conclude that $D_n'=D_n$ as desired.

Since $A$ and $L$ differ be a translation of an element in $2\Z^2$, every $n$-parallelogram with respect to $A$ is an  $n$-parallelogram with respect to $L$, and vice versa. So we may assume that $A=L$.
Without loss of generality, we may assume that $D_h\leq D_v$, so that $D_n'=D_h$. Observe that we need at least $m$ parallelograms to connect a pair of opposite sides of an $(m\times m)$-grid of parallelograms.
Notice that
\begin{eqnarray} \label{lma}
L^{-n}([-m,m]\times[-m,m])\cap(-\infty,+\infty)\times \{1\}\not=\emptyset
\end{eqnarray}
if and only if there exist $m$ $n$-parallelograms connecting $(-\infty,+\infty)\times\{0\}$ and $(-\infty,+\infty)\times\{1\}$.
Hence, $D_n'$ is equal to the smallest positive integer $m$ such that  $y_0=\max\{y_{(\pm m,\pm m)}\}$ is greater than 1, where
\[\left(
  \begin{array}{c}
    x_{(\pm m,\pm m)}\\
    y_{(\pm m,\pm m)} \\
  \end{array}
\right)
=L^{-n}
\left(
  \begin{array}{c}
    \pm m\\
    \pm m \\
  \end{array}
\right),
\]
and $(\pm m, \pm m)$ varies over $(m,m), (m,-m), (-m,m)$ and $(-m,-m)$.
 Since the image of $\{v\: \|v \|_\infty  = 1\} = \partial [-1,1]^2$ under $L^{-n}$ is the boundary of a parallelogram,
\[\|L^{-n}\|_{\infty}= \max \left\|L^{-n}
\left(
  \begin{array}{c}
    \pm 1\\
    \pm 1 \\
  \end{array}
\right)\right\|_{\infty}, \]
so
\[D_n'\|L^{-n}\|_{\infty}= \max \left\|L^{-n}
\left(
  \begin{array}{c}
    \pm D_n'\\
    \pm D_n' \\
  \end{array}
\right)\right\|_{\infty}=y_0\geq 1.\]
Hence,
\[\frac{1}{\|L^{-n}\|_{\infty}}\leq D_n'\leq \frac{1}{\|L^{-n}\|_{\infty}}+1.\]
Since $D_n=D_n'$, the proof is complete.
\end{proof}

\begin{cor} \label{combexpansionfactor}
Let $f$ be a Latt\`es-type map with orbifold type $(2,2,2,2)$, and let $A$ be its affine map from $\R^2$ to $\R^2$.
Then the combinatorial expansion factor $\Lambda_0(f)$ equals the minimum absolute value of the eigenvalues of $A$.
\end{cor}

\begin{proof}
Let $L$ be the linear map of $A$.
By the previous proposition,
\[\frac{1}{\|L^{-n}\|_{\infty}}\leq D_n\leq \frac{1}{\|L^{-n}\|_{\infty}}+1.\]
Taking $n$-th roots gives
\[\left(\frac{1}{\|L^{-n}\|_{\infty}}\right)^{1/n}\leq {D_n}^{1/n}\leq \left(\frac{1}{\|L^{-n}\|_{\infty}}+1\right)^{1/n},\]
so by Gelfand's formula (see Theorem 13 in \cite[Chapter 8]{LaxLinear}),
\[\lim_{n\ra\infty}{D_n}^{1/n}=\lim_{n\ra\infty}\frac{1}{\|L^{-n}\|_{\infty}^{1/n}}=\frac{1}{\rho(L^{-1})},\] where $\rho(L^{-1})$ is the spectral radius of $A^{-1}$. On the other hand, the spectral radius of $L^{-1}$ is the maximal absolute value of the eigenvalues of $L^{-1}$, which is equal to $1/|\lambda_1|$, where $|\lambda_1|$ is the minimum absolute value of the eigenvalues of $A$ (and $L$). We conclude that
\[\Lambda_0(f)=\lim_{n\ra\infty}D_n^{1/n}=|\lambda_1|.\]
\end{proof}


\begin{pro} \label{opinequality}
Let $f$ be a Latt\`es map over $\L$ and let $\mathcal C\subset \S^2$ be a Jordan curve containing all postcritical points of $f$. Then there exists a constant $c>0$ such that $D_n(f,\mathcal C)\geq c\,(\deg f)^{n/2}$ for all $n>0$.
\end{pro}

\begin{proof}
Let $f$ be a Latt\`es map induced by the linear map $L\:\C\ra \C$ defined by $z\mapsto \lambda z$.
It follows from the claim at the end of the proof of Lemma \ref{lattesexpanding1} that
\begin{equation}\label{e:degree}
\deg f = |\lambda|^{2}.
\end{equation}

First, assume that $f$ is a Latt\`es map with orbifold type $(2,2,2,2)$. By Proposition \ref{dnrelation} and \eqref{e:degree},
\[D_n\geq \frac1{\|L^{-n}\|_{\infty}}=\frac{1}{|\lambda^{-n}|}=|\lambda|^{n}=(\deg f)^{n/2}.\]

For Latt\`es maps with $\#\post(f)=3$, there exists a Jordan curve $\mathcal{C}$ containing the postcritical set that lifts to a
tiling of the plane by Euclidean triangles. We refer the reader to Page 13 in \cite{MilLattes} for more details.
We will call the triangles in the tiling above \emph{unit triangles}.


The idea of the proof is similar to the proof of Proposition \ref{dnrelation}. That is, we will attempt to lift everything to $\C$. Since a unit triangle is holomorphic to a $0$-tile, $D_n(f,\mathcal{C})$ is the same as the number of pre-images of the unit triangles under $L^n$ needed to connect all edges of the unit triangle. We present the details below.

The pre-images of the unit triangles under $L$ are triangles similar to the unit triangle by a ratio of $|\lambda|^n$, which we call \emph{$n$-triangles}. Since the image of a $1$-triangle under $\Theta \circ L=f\circ \Theta$ is a $0$-tile 
(see the commutative diagram below), and
since the only connected set that maps onto a $0$-tile under $f$ is a $1$-tile, we conclude that the image of a $1$-triangle under $\Theta$ is a $1$-tile.
\[\xymatrix{
  \C \ar[d]_{\Theta} \ar[r]^{L}
                & \C \ar[d]^{\Theta}  \\
  \hat{\C}  \ar[r]_{f}
                & \hat{\C} .            }
\]
Similarly, the image of an $n$-triangle under $\Theta$ is an $n$-tile since the only connected set that maps onto a $0$-tile under $f^n$ is a $n$-tile.

Let $D_n'$ be the minimum number of $n$-triangles needed to connect all three edges of a $0$-triangle $\triangle$. Assume that $D_n$ $n$-tiles connect the three edges of the $0$-tile $T=\Theta(\triangle)$.
We claim that $D_n= D_n'$.
Let $T_1,T_2,\ldots, T_{D_n}$ be a sequence of $n$-tiles connecting the three edges of the 0-tile $T$. Let $T_1'$ be the connected component of $\Theta^{-1}(T_1)$ intersecting with $\triangle$, which is an $n$-triangle. Let $T_2'$ be the connected component of $\Theta^{-1}(T_2)$ intersecting with $T_1'$, which is also an $n$-triangle. Let $T_3'$ be the connected component of $\Theta^{-1}(T_3)$ intersecting with $T_2'$, and so on. We obtain a sequence of $n$-triangles $T_1',\ldots, T_{D_n}'$ connecting the three edges of $\triangle$, and hence $D_n'\leq D_n$.
On the other hand, suppose that a sequence of $n$-triangles $T_1,\ldots, T_{D_n'}$ connects the three edges of $\triangle$. Then the sequence $\Theta(T_1),\ldots, \Theta(T_{D_n'})$ of $n$-tiles connects the three edges of the 0-tile $T$. We conclude that $D_n'=D_n$ as desired.

Let $R$ denote the length of the longest edge of the unit triangle. Then the longest edge of an $n$-triangle is $R|\lambda|^{-n}$.
Let $S$ be the union of all the edges of the $n$-triangles $T_1, \ldots T_{D_n}$. Then $S$ is a union of line segments and hence we may consider its length.
On the one hand, the length of $S$ is bounded above by $3D_nR|\lambda|^{-n}$, since $R|\lambda|^{-n}$ is the longest edge of an $n$-triangle. On the other hand, the length of $S$ is bounded below by the length $r$ of the shortest connected union of line segments connecting the three edges of a unit triangle.
Explicitly, $r$ is the minimum length of the line joining the longest edge of a unit triangle to the vertex containing the other two edges. Using \eqref{e:degree}, we conclude that
\[D_n\geq \frac{r}{3R|\lambda|^{-n}}=\frac{r}{3R}|\lambda|^n= \frac{r}{3R}(\deg f)^{n/2}.\]
\end{proof}

\section{Existence of the Visual Metric} \label{existence}
\noindent
In this section, we prove that there exists a visual metric on $\S^2$ with expansion factor equal to $\deg(f)^{1/2}$ under the three conditions in Theorem \ref{main}. This will imply that the expanding Thurston map $f$ is topologically conjugate to a Latt\`es map.

We refer to the following assumptions as $(*)$:
\[(*)\quad \begin{array}{l} \label{*}
  \hbox{The map $f \:\S^2\ra\S^2$ is an expanding Thurston map with no} \\
  \hbox{periodic critical points, and $\mathcal C$ is a Jordan curve in $\S^2$ that }\\
  \hbox{is invariant under $f$ and satisfies $\post(f)\subset \mathcal C$. }\\
\end{array}\]
Notice that the cell decompositions $\D^n(f,\mathcal C)$ of $\S^2$ induced by a Jordan curve as in $(*)$ are compatible with one another in the sense that $\D^{n+1}(f,\mathcal C)$ is a subdivision of $\D^n(f, \mathcal C)$.

Let $\lambda_0:=(\deg f)^{1/2}$. We refer to the following assumptions as $(**)$:
\[(**)\quad \begin{array}{l} \label{**}
  \hbox{The map $f \:\S^2\ra\S^2$ is an expanding Thurston map with no} \\
  \hbox{periodic critical points, and there exists a constant $c>0$ such}\\
  \hbox{that $D_n=D_n(f,\mathcal C)\geq c\lambda_0^{n}$ for all $n>0$, where $\mathcal C$ is a Jordan }\\
  \hbox{curve in $\S^2$ that is invariant under $f$ and satisfies $\post(f)\subset \mathcal C$. }\\
\end{array}\]

First, let us review some definitions (see the proof of Theorem 17.3 in \cite{BMExpanding} for more details). Let $f$ be an expanding Thurston map. By Section \ref{expanding}, we have a sequence of cell decompositions of the underlying space $\S^2$ by tiles. We define a \emph{tile chain} $P$ to be a finite sequence of tiles
\[X_1,\ldots, X_N\]
such that $X_j\cap X_{j+1}\not=\emptyset$ for $j=1,\ldots,N-1$. We also write
\[P=X_1X_2\ldots X_N,\]
and we use $|P|$ to denote the underlying set $\bigcup_{i=1}^N X_i$ with the subspace topology. In addition, if $X_n$ intersects with $X_1$, then we call the tile chain $P$ a \emph{tile loop}. For $A, B\subseteq \S^2$, we say that the tile chain $P$ \emph{joins} the sets $A$ and $B$ if
\[A\cap X_1\not=\emptyset \mbox{ and } B\cap X_N\not=\emptyset.\]
We say that the tile chain $P$ \emph{joins} the points $x$ and $y$ if $P$ joins $\{x\}$ and $\{y\}$. A \emph{subchain} of $P=X_1X_2\ldots X_N$ is a tile chain of the form
\[X_{j_1}\ldots X_{j_s}, \mbox{ where }1 \leq j_1 <\ldots < j_s \leq N.\]
We call a tile chain $P=X_1X_2\ldots X_N$ \emph{simple} if there is no subchain of $P$ that joins $X_1$ and $X_N$. We call a tile chain $P=X_1X_2\ldots X_N$ an \emph{$n$-tile chain} if all the tiles $X_i$ are $n$-tiles $1\leq i\leq N$. An $n$-tile chain $P=X_1X_2\ldots X_N$ is called an \emph{$e$-chain} if there exists an $n$-edge $e_i$ with $e_i\subseteq X_i\cap X_{i+1}$ for $i = 1,\ldots, N$. The $e$-chain \emph{joins} the tiles $X$ and $Y$ if $X_1 = X$ and $X_N = Y$. A set $M$ of $n$-tiles is \emph{$e$-connected} if every two tiles in $M$ can be joined by an $e$-chain consisting of $n$-tiles contained in
$M$.

The following lemma is from \cite[Lemma 14.4]{BMExpanding}.
\begin{lem} \label{econncted}
Let $\gamma\subset \S^2$ be a path in $\S^2$ defined on a closed
interval $J\subset \R$ and $M = M(\gamma)$ be the set of tiles having nonempty
intersection with $\gamma$. Then $M$ is e-connected.
\end{lem}

If $P=X_1\ldots X_N$ is a tile chain, then we define the \emph{length} of the tile chain to be the number of tiles in $P$:
\[\length (P)=N.\]
For $n\geq 1$, we define a function
\begin{eqnarray} \label{defdn}
d_n\:\S^2\times\S^2\ra \R
\end{eqnarray}
as follows: for any $x,y\in\S^2$,
if $x=y$, then $d_n(x,y)=0$;
otherwise,
\begin{eqnarray*}
d_n(x,y)=\min \{\length(P)\}\lambda_0^{-n},
\end{eqnarray*}
where the minimum is taken over all $n$-tile chains $P$ joining $x$ and $y$.
It is clear that $d_n$ is a metric on $\S^2$.

In the following, we will show that for any $x,y\in \S^2$ with $x\not=y$, the ratio
\[d_n(x,y)/ \lambda_0^{-m(x,y)}\]
has a uniform upper and lower bound for all $n>m(x,y)$, where $m(x,y)=m_{f,\mathcal C}(x,y)$ is defined in Definition \ref{defofm} (see Lemma \ref{leq} and Lemma \ref{geq}). Then we will define a distance function
\[d=\limsup_{n\ra\infty}d_n,\]
and we will see that this metric $d$ is a visual metric on $\S^2$ with expansion factor $\Lambda$ (see Proposition \ref{existenceofvisualmetric}).

\begin{de}
For $n\geq 3$, we call a topological space $X$ an \emph{$n$-gon} if $X$ is homeomorphic to the closed unit disk $\overline{\mathbb D}\subset \R^2$ with $n$ points marked on the boundary of $X$. Since the boundary of an $n$-gon is homeomorphic to $\S^1$, there is a natural cyclic order for the $n$ marked points on the boundary. We call these $n$-points \emph{vertices} of the $n$-gon and the parts of the boundary of $X$ joining two consecutive vertices in the cyclic order the \emph{edges} of the $n$-gon.
\end{de}

Now let us review some basic definitions from graph theory (we refer the reader to \cite{DieGraph} for more details). A \emph{graph} $G$ is a pair $(V,E)$ of sets such that the \emph{edge set} $E=E(G)$ is a symmetric subset of the Cartesian product $V \times V$ of the \emph{vertex set} $V=V(G)$. We call a graph $G'=(E',V')$ a \emph{subgraph} of $G=(V,E)$ if
\[E'\subseteq E \mbox{ and } V'\subseteq V,\] written as $G'\subseteq G$.
A \emph{(simple) path} in a graph $G=(V,E)$ is a non-empty subgraph $P=(V',E')$ of the form
\[V'=V'(P)=\{x_0,x_1,\ldots,x_k\}\] and
\[E'=E'(P)=\{x_0x_1,x_1x_2, \ldots, x_{k-1}x_k\},\]
where the $x_i\in V'$ are all distinct, and $uv$ denotes the edge with end points $u,v\in V$. We also write a path as
\[P=x_0x_1\ldots x_k\]
and call $P$ a path from $x_0$ to $x_k$. Given sets $A,B$ of vertices in $G$, we call $P=x_0x_1\ldots x_k$ an \emph{$A$-$B$ path} if $x_0\in A $ and $x_k\in B$.
Given a graph $G=(V, E)$, if $A, B, X\subseteq V$ are such that $X$ is disjoint from $A$ and $B$, and every $A$-$B$ path in $G$ contains a vertex from $X$, we say that $X$ \emph{separates} the sets $A$ and $B$ in $G$. We call $X$ a \emph{separating set} for $A$ and $B$ in the graph $G$.  We will use the following theorem (see Theorem 3.3.1 in \cite{DieGraph}). In general,
given a topological space $T$, let $A,B\subset T$ such that
\[A\cap B=\emptyset.\]
We say that a set $U\subset T$ \emph{separates} $A$ and $B$ if for any path $\gamma\subseteq T$ joining $A$ and $B$,
\[\gamma \cap U\not=\emptyset.\] We call $U$ a \emph{separating set} of $A$ and $B$ in $T$. For $x,y\in T$, we call $U$ a \emph{separating set} of $x$ and $y$ in $T$ if $U$ separates $\{x\}$ and $\{y\}$ in $T$.

\begin{thm}[Menger's theorem] \label{mengers}
Let $G=(V,E)$ be a finite graph and $A, B\subseteq V$. Then the minimal cardinality of a set separating $A$ and $B$ in $G$ is equal to the maximal number of pairwise disjoint $A$-$B$ paths in $G$.
\end{thm}

Let $X$ be a set of $m$-tiles, and denote the union by $|X|$. All the vertices and edges of $m$-tiles in $X$ give a cell decomposition of $|X|$. We define a graph $G(X)$ with vertex set being the set of all $m$-tiles in $X$, and with an edge between two vertices if and only if the corresponding $m$-tiles share a common edge. We call $G(X)$ the \emph{dual graph} associated with $X$. Note that a subset of $X$ is $e$-connected if and only if the corresponding vertex set in $G(X)$ is path connected.

Given an $l$-vertex $v$,
recall that $W^l(v)$ is the union of the interior of $n$-cells intersecting with $v$, so $W^l(v)$ is connected. For $n>l$, let $\D^n(v)$ be the set of $n$-cells in $W^l(v)$. This gives us a cell decomposition of $W^l(v)$. Let $G^n(v)$ be the dual graph associated with the cell decomposition $\D^n(v)$.

\begin{lem} \label{congraph}
Assume that $(*)$ holds and $n>l\geq 0$. Then with the notation above, the graph $G^n(v)$ is path connected.
\end{lem}

\begin{proof}
By Lemma \ref{flower}, the $l$-flower $W^l(v)$ is simply connected. There exists a path $\gamma\subset W^l(v)$ containing all $n$-vertices in $W^l(v)$. Let $V = V(\gamma)$ be the set of tiles having nonempty
intersection with $\gamma$. Then $V$ is the vertex set of the graph $G^n(v)$.
Lemma \ref{econncted} states that $V$ is $e$-connected, which implies that $G^n(v)$ is path connected.
\end{proof}

Note that we do no need $f$ to be expanding in Lemma \ref{congraph}.


\begin{lem} \label{Usep}
Assume that $(*)$ holds and $n>0$. Let $X$ be a set of $n$-tiles, and $A, B, S\subset X$. If $S$ separates $A$ and $B$ in the graph $G(X)$, then $|S|$ separates $|A|$ and $|B|$ in $|X|$.
\end{lem}


\begin{proof}
Assume that the set $|S|$ does not separate $|A|$ and $|B|$ in $|X|$. Then there exists a path $\gamma\subset |X|$ joining $|A|$ and $|B|$ such that
\[\gamma\cap |S|=\emptyset.\]
By Lemma \ref{econncted}, the set $M(\gamma)$ of $n$-tiles intersecting with $\gamma$ is $e$-connected.
In addition, we have that
\[M(\gamma)\cap S=\emptyset\] since $\gamma\cap |S|=\emptyset$.
This means that there exists an $e$-path in $M(\gamma)\subset X\setminus S$ joining $A$ and $B$. This is a contradiction of the definition of the separating set $S$.
Hence, the set $|S|$  separates $|A|$ and $|B|$ in $|X|$.
\end{proof}

Continuing with the notation of Lemma \ref{Usep} above, if in addition, $S$ is a minimal separating set of $A$ and $B$ in the graph $G(X)$, we would like to show that $|S|$ is a connected set. This will follow from Lemma \ref{Uconnected}. In order to prove it, we need some preparations.

The following theorem is from \cite[Page 110]{NewElements}.

\begin{thm}[Janiszewski] 
Let $A$ and $B$ be closed subset of $\S^2$ such that $A\cap B$ is connected. If neither $A$ nor $B$ separates two points $x$ and $y$ in $\S^2$, then $A\cup B$ does not separate $x$ and $y$ either.
\end{thm}

As a corollary of Janiszewski Theorem, we have the following.
\begin{cor} \label{Janiszewski}
Let $U$ be a closed subset of $\S^2$ with finitely many connected components. For two path-connected regions $X,Y\subset \S^2$ which are disjoint from $U$, if the set $U$ separates $X$ and $Y$, then one of the connected components of $U$ separates $X$ and $Y$.
\end{cor}

\begin{proof}
Fix $x\in X$ and $y\in Y$. By induction on the number of connected components of $U$ and by Janiszewski's Theorem,
there exists a connected component $U'$ of $U$ that separates $x$ and $y$. Consider a path $\gamma$ connecting points
$x' \in X$ and $y' \in Y$. Let $\alpha \subset X$ be a path from $x$ to $x'$, and let $\beta \subset Y$ be a path from $y'$ to $y$. Then the path $\alpha\gamma\beta$ joining $x$ and $y$ intersects $U'$. Hence, the path $\gamma$ intersects $U'$, and $U'$ separates $x'$ and $y'$. We conclude that $U'$ separates $X$ and $Y$.
\excise{
If $X$ and $Y$ contain only one point respectively, then it follows from Janiszewski's Theorem by induction on the number of connected components of $U$.
If $X$ or $Y$ has more than one point, without loss of generality we may assume that $X$ has more than one point.
Assume that none of the connected components of $U$ separates $X$ and $Y$, i.e., there exist $x,x'\in X$ with $x\not=x'$, and $y,y'\in Y$, such that $x$ and $y$ separates by $U_1$ but not $U_2$, and $x'$ and $y'$ separated by $U_2$ but not $U_1$, where $U_i$ for $i=1,2$ is a union of connected components of $U$ and $U_1\cap U_2=\emptyset$.
Let $\alpha \subset X$ be a path from $x'$ to $x$, and let $\beta \subset Y$ be a path from $y$ to $y'$. Let $\gamma\subset \S^2$ be a path from $x$ to $y$.

so $\gamma$ intersects with $U_1$. Then the path $\alpha\gamma\beta$ joining $x'$ and $y'$ also intersects with $U_1$ which is a contradiction.
}
\end{proof}

\begin{lem} \label{Uconnected}
Let $W$ be a simply connected region in $\S^2$.
Let $U$ be a closed subset of the closure $\overline{W}$ of $W$ in $\S^2$ with finitely many connected components. For two path-connected regions $X,Y\subset W$ which are disjoint from $U$, if $U$ separates $X$ and $Y$ in $W$, then there exists a connected component of $U$ separating $X$ and $Y$.
\end{lem}

\begin{proof}
Without loss of generality, we may assume that $U=\cup_{i=1}^{I}U_i$ where $U_i$ is a connected components of $U$, and $U_i\cap U_j=\emptyset$ if $i\not=j$. Let $\partial W$ be the boundary of $W$ in $\S^2$. For any path $\gamma\subseteq \S^2$ from $X$ and $Y$, if $\gamma\subset W$, then
\[\gamma\cap U\not=\emptyset,\]
and if $\gamma \not\subset W$, then
\[\gamma \cap \partial W\not=\emptyset.\]
So the set $U\cup \partial W$ separates $X$ and $Y$ in $\S^2$.
\begin{description}
  \item[Case 1] None of the $U_i$ intersect with $\partial W$. Since $\partial W$ does not separate $x$ and $y$ in $\S^2$, Corollary \ref{Janiszewski} implies that one of the $U_i$ separates $X$ and $Y$ in $\S^2$.
  \item[Case 2] All of the $U_i$ intersect with $\partial W$. Let
      \[U_i'=U_i\cup \partial W \mbox{ for } 1\leq i\leq I\]
      Notice that $U_i'\cap U_j'=\partial W$ is connected for any $i\not=j$. We claim that one of the $U_i'$ separates $X$ and $Y$ in $\S^2$. If none of $U_i'$ separates $X$ and $Y$ in $\S^2$, then by Janiszewski's Theorem, the set \[\cup_{i=1}^I U_i'=U\cup \partial W\] does not separate $X$ and $Y$, which is a contradiction.
      Without loss of generality, assume that $U_1'$ separates $X$ and $Y$ in $\S^2$. Then $U_1$ separates $X$ and $Y$ in $W$.
  \item[Case 3] Only some of the $U_i$ intersect with $\partial W$. Without loss of generality, assume that
      \[U_i\cap \partial W=\emptyset \mbox{ for } 1\leq i\leq J<I,\]and
      \[U_i\cap \partial W\not=\emptyset \mbox{ for } J< i\leq I,\]
      Let $U'=\cup_{i=J+1}^{I} U_i\cup \partial W$. By Corollary \ref{Janiszewski}, either one of the $U_i$ for $1\leq i\leq J$ or $U'$ separates $X$ and $Y$ in $\S^2$. If one of the $U_i$ for $1\leq i\leq J$ separates $X$ and $Y$, we are done. If $U'$ separates $X$ and $Y$, then it is Case 2. This implies that one of the $U_i$ for $i\in I$ separates $X$ and $Y$ in $W$.
\end{description}
Hence,  one of the connected components of $U$ separates $X$ and $Y$ in $W$.
\end{proof}

\begin{pro} \label{other}
Let $f$ be a Thurston map without periodic critical points and let $\mathcal C\subset \S^2$ be a Jordan curve that is invariant under $f$ and contains ${\rm post}(f)$. Then there exists a constant $C>0$ such that
\[D_n=D_n(f,\mathcal{C})\leq C\deg(f)^{n/2} \] for all $n\geq 0$.
\end{pro}

\begin{proof}

First, assume that $m=\#{\rm post}(f)\geq 4$. Let $e_1,\ldots, e_m$ be the $0$-edges in cyclic order. Fixing a $0$-tile, let $X$ be the union of all $n$-tiles in this $0$-tile, and let $G(X)$ be the dual graph of $X$. Let $A$ be the set of all $n$-tiles in $X$ intersecting with $e_1$, and $B$ be the set of all $n$-tiles in $X$ intersecting with $e_3$.

Let $S$ be a minimal separating set between $A$ and $B$ in $G(X)$. By Lemma \ref{Usep}, the set $|S|$ separates $|A|$ and $|B|$ in $|X|$. Consider the subspace topology on $|A|$ and $|B|$. Since $e_1\subset \inter(|A|)$ and $e_3\subset \inter(|B|)$, $e_1$ and $e_3$ are both connected and disjoint from $|S|$. So
by Lemma \ref{Uconnected}, one of the connected components of $|S|$ separates $e_1$ and $e_3$. Since $S$ is a minimal separating set of $A$ and $B$, the separating set $|S|$ is connected.

Notice that there are two connected components in
\[Q=\partial |X|\setminus \big(\inter(e_1)\cup \inter(e_{3})\big),\]
which gives us two disjoint paths from $e_1$ to $e_3$.
Since $|S|$ intersects with both components of $Q$, the set $|S|$ joins at least two disjoint $0$-edges. Hence, there exist at least $D_n$ $n$-tiles in $S$.

By Menger's theorem, there are at least $D_n$ many disjoint $A$-$B$ paths. Let $N_n$ be the minimum number of tiles in an $A$-$B$ path, and since an $A$-$B$ path is an $n$-tile chain joining opposite sides of the Jordan curve $\mathcal C$, we have $D_n\leq N_n$. We get
\[D_n^2\leq D_nN_n\leq 2(\deg(f))^n,\]so
\[D_n\leq C\deg(f)^{n/2}\]for $C=\sqrt2$.

When $\#{\rm post}(f)=3$, we can cut along any two edges of the $3$-gons, and we unfold it to get a $4$-gon. Let $X$ be the union of all $n$-tiles in this $4$-gon, and pick two non-adjacent edges in this $4$-gon and call them $e_1$ and $e_3$. Now we can apply the same argument as in the case when $\#{\rm post}(f)=4$ above.
\end{proof}

Given an $h$-tile $X^h$, and an $(h-1)$-tile $X^{h-1}$, if $X^h\subset X^{h-1}$, then we call $X^{h-1}$ the \emph{parent} of $X^h$.

\begin{lem} \label{short}
Assume that $(**)$ holds and $n>h> 0$.
Let $X^h, Y^h$ be $h$-tiles and let $X^{h-1}, Y^{h-1}$ be their parents respectively. Assume $X^{h-1}\cap Y^{h-1}\not=\emptyset$. Then there exists an $n$-tile chain with at most $c'\lambda_0^{n-h}$ tiles joining $X^h$ and $Y^h$, where $c'>0$ only depends on $f$.
\end{lem}

\begin{proof}
The lemma is trivial if $X^h\cap Y^h\not=\emptyset$. Now assume that $X^h$ and $Y^h$ are disjoint.
Let $v$ be an $(h-1)$-vertex in $X^{h-1}\cap Y^{h-1}$, and let $G^{n}(v)$ and $G^h(v)$ be the dual graphs associated to the cell decompositions of $W^{h-1}(v)$ consisting of $n$-tiles and $h$-tiles respectively (see the paragraph before Lemma \ref{congraph} for the meaning of the notation). By Lemma \ref{congraph}, the graphs $G^{n}(v)$ and $G^h(v)$ are both path-connected.

Let $A$ be the set of all $n$-tiles in $X^h$, and $B$ be the set of all the $n$-tiles in $Y^h$. Let $S$ be a minimal separating set between $A$ and $B$. 
Consider $A,B,S$ as vertices in $G^n(v)$.
By Lemma \ref{Usep}, the set $|S|$ separates $X^h$ and $Y^h$ in $W^{h-1}(v)$. Since $\inter(X^h)$ and $\inter(Y^h)$ are both connected regions and disjoint from $|S|$, by Lemma \ref{Uconnected}, one of the connected component of $|S|$ separates $\inter(X^h)$ and $\inter(Y^h)$. Since $S$ is a minimal separating set, the separating set $|S|$ is connected.


\vspace{.5cm}
\begin{center}
\mbox{ \scalebox{0.7}{\includegraphics{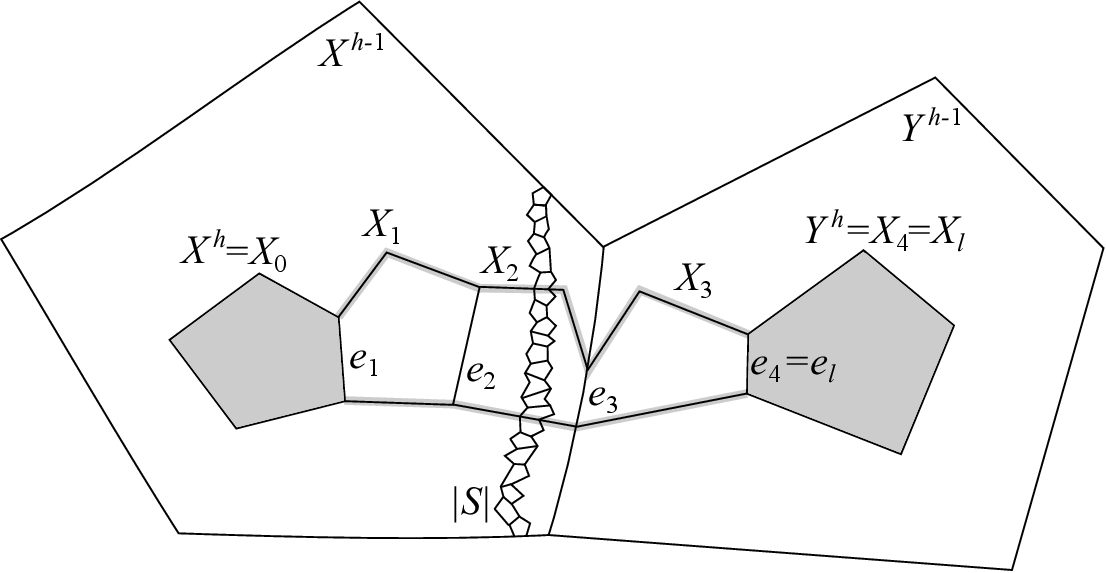}}}
\end{center}
\vspace{.5cm}

Since $G^h(v)$ is path-connected, there is an $e$-chain $P=X_0X_1X_2\ldots X_l$ of $h$-tiles in $W^{h-1}(v)$ with $X_0=X^h$ and $X_l=Y^h$. After possibly replacing $P$ with a shorter e-chain, we may assume that $X_i\not=X_j$ if $i\not=j$. Pick an $h$-edge in $X_{i-1}\cap X_{i}$ and call it $e_{i}$, for $i=1,2,\ldots, l$. Notice that there are two connected components in
\[Q_i=\partial X_i\setminus \big(\inter(e_i)\cup \inter(e_{i+1})\big),\]
for $i=1,\ldots, l-1$.
The union $Q=\cup_{i=1}^{l-1} Q_i$ has two connected components, which gives us two disjoint paths from $X^h$ to $Y^h$ (See the figure above).
Since $|S|$ intersects with both components of $Q$, the set $|S|$ joins at least two disjoint $h$-edges. Hence, there exist at least $D_{n-h}$ $n$-tiles in $S$.

Let $N_n$ be the minimal number of $n$-tiles in an $A$-$B$ path. By Menger's Theorem, there are at least $D_{n-h}$ non-disjoint $A$-$B$ paths in $G^n(v)$. Thus
\[N_nD_{n-h}\leq K(\deg f)^{(n-h)},\]
where $K>0$ is a constant as in Lemma \ref{noperiodic} which only depends on $f$. Hence
\[ N_n\leq \frac{K(\deg f)^{(n-h)}}{D_{n-h}}\leq \frac{K}{C} \lambda_0^{n-h}=c'\lambda_0^{n-h},\]where $c'$ only depends on $f$.
\end{proof}

Recall the function $d_n$ as defined in \ref{defdn}.

\begin{lem} \label{leq}
Assume that $(**)$ holds and assume that $\lambda_0 >2$.
There exists a constant $C>0$ depending only on $f$, such that for any $x,y\in \S^2$ with $x\not=y$ and for any $n> m(x,y)$,
\[d_n(x,y)\leq C\lambda_0^{-m(x,y)}.\]
\end{lem}

\begin{proof}
For simplicity of notation, let $m=m(x,y)$. With the notation of Lemma~\ref{short}, let
\[
A(k) = 2^{k+1}\lambda_0^2 + c'\lambda_0^k(1 + 2/\lambda_0 + \cdots + (2/\lambda_0)^{k-1}),
\]
for all non-negative integers $k$, and let $A(k) = 2$ for all negative integers $k$. 
Observe that $A(k) = 2A(k - 1) + c'\lambda_0^k$ for $k > 0$, and $A(k) \ge A(k-1)$ for all $k$.
We will first show that there exists an $n$-tile chain joining $x$ and $y$ of length at most $A(n-m)$. 

By the definition of $m$, $m > 0$ and there exists non-disjoint $(m - 1)$-tiles $X^{m-1}$ and $Y^{m-1}$ containing $x$ and $y$ respectively.
If $n < m$, then there exists an $n$-tile chain of length $A(n-m) = 2$ joining $x$ and $y$. 
If $n = m$, then since an $(n-1)$-tile contains precisely $\lambda_0^2$ $n$-tiles,
the union of all $n$-tiles in $X^{m-1}$ and $Y^{m-1}$ forms an $n$-tile chain joining $x$ and $y$ of length $A(0) = 2\lambda_0^2$.

Hence we may assume that $n > m$. We will argue by induction on $n - m$. 
Fix $m$-tiles $X^m \subseteq X^{m-1}$ and $Y^{m} \subseteq Y^{m-1}$ containing $x$ and $y$ respectively.
By Lemma~\ref{short}, there exists an $n$-tile chain $P$ of length at most $c'\lambda_0^{n-m}$ joining
$X^m$ and $Y^{m}$. Let $x' \ne x$ and $y' \ne y$ be points in the first and last $n$-tiles of the chain $P$ respectively.
Then any $m$-tile
containing $x'$ also contains the first $n$-tile in $P$, and hence has non-empty intersection with $X_m$. Therefore, $m(x,x')> m$ and by induction
there exists an $n$-tile chain $P_x$ joining $x$ and $x'$ of length at most $A(n-m(x,x')) \le A(n - m - 1)$. Similarly, 
there exists an $n$-tile chain $P_y$ joining $y$ and $y'$ of length at most $A(n - m - 1)$.  We conclude that the union of $P_x$, $P_y$ and 
$P$ is an $n$-tile chain joining $x$ and $y$ of length at most 
\[
2A(n-m-1) + c'\lambda_0^{n-m} = A(n-m).
\]

Finally, for $k > 0$, since $\lambda_0 > 2$, we have 
\[
A(k) < 2^{k+1}\lambda_0^2 + c'\lambda_0^k/(1 - 2/\lambda_0) < \lambda_0^k[4 + c'\lambda_0/(\lambda_0 - 2)].
\] 
Hence the result follows by setting $C = 4 + c'\lambda_0/(\lambda_0 - 2)$. 
\end{proof}

\begin{lem} \label{geq}
Assume that $(**)$ holds. For any $x,y\in \S^2$ with $x\not=y$, and for any $n>m(x,y)$, we have
\[d_n(x,y)\geq c\lambda_0^{-m(x,y)},\]where $c>0$ is the same constant as in $(**)$.
\end{lem}

\begin{proof}
Let $m=m(x,y)$, and let $X^m$ and $Y^m$ be disjoint $m$-tiles containing $x$ and $y$ respectively. The length of any $n$-tile chain joining $X^m$ and $Y^m$ is at least $D_{n-m}$. Hence, we have that
\begin{eqnarray*}
d_n(x,y)&\geq & D_{n-m}\lambda_0^{-n}\geq c\lambda_0^{n-m}\lambda_0^{-n}=c\lambda_0^{-m}=c\lambda_0^{-m(x,y)}.
\end{eqnarray*}
\end{proof}

\begin{pro} \label{existenceofvisualmetric}
Let $f\:\S^2\ra \S^2$ be an expanding Thurston map with no periodic critical points. Assume there exists $c>0$ such that $D_n=D_n(f,\mathcal C)\geq c\,(\deg f)^{n/2}$ for all $n>0$, where $\mathcal C$ is a Jordan curve containing $\mbox{post}(f)$. Then there exists a visual metric with
$\Lambda=(\deg f)^{1/2}$ as the expansion factor.
\end{pro}
See Definition \ref{visual}, for the definition of a visual metric.

\begin{proof}
By Theorem \ref{invariantJordancurve}, for some $n>0$, there exists a Jordan curve $\mathcal C$ containing ${\rm post}(f)$ that is invariant under $f^n$. Proposition 8.8 (v) in \cite{BMExpanding} states that a metric is a visual metric for $f^n$ if and only if it is a visual metric for $f$. Hence, we may assume that there exists a Jordan curve $\mathcal C$ that is invariant under $f$. Since we can pass to an iterate of $f$, we may assume that
\[\lambda_0=(\deg f)^{1/2} >2.\]

Let
\[d=\limsup_{n\ra\infty} d_n,\]
where $d_n$ is defined in equation \eqref{defdn}. We will show that $d$ is a visual metric on $\S^2$ with expansion factor $\Lambda_0$.

Fix $x,y\in S^2$ such that $x\not=y$. By Lemma \ref{leq},
\[d(x,y) =\limsup_{n\ra \infty}d_n(x,y) \leq C\lambda_0^{-m(x,y)},\]
where $C>0$ only depends on $f$.  By Lemma \ref{geq},
\[d(x,y) =\limsup_{n\ra \infty}d_n(x,y) \geq \lambda_0^{-m(x,y)}.\]

In addition, the function $d$ is a metric since $d_n$ is metric on $\S^2$ for all $n> 0$.
Therefore, the function $d$ is a visual metric on $\S^2$ with expansion factor $\Lambda=(\deg f)^{1/2}$.
\end{proof}

\section{The Sufficiency of the Conditions} \label{sufficiency}
\noindent
In this section, we show that under the conditions in Theorem \ref{main}, the expanding Thurston map $f$ is topologically conjugate to a Latt\`es map.

For the next definition, we use the notion of \emph{continuum}, which is a compact connected set consisting of more than one point.
\begin{de}
A metric space $(X,d)$ is called \emph{linearly locally connected} (denoted \emph{LLC}) if there exists some $\lambda> 1$ such that the following two conditions are satisfied:
\begin{description}
  \item[(LLC1)] If $B(a,r)$ is a ball in $X$ and $x, y\in B(a, r)$ and $x\not=y$, then there exists
a continuum $E\subseteq B(a,\lambda r)$ containing $x$ and $y$;
  \item[(LLC2)] If $B(a,r)$ is a ball in $X$ and $x, y \in X\setminus B(a, r)$ and $x\not=y$, then there exists
a continuum $E \subseteq X \setminus B(a, r/\lambda)$ containing $x$ and $y$.
\end{description}
\end{de}

A metric space $X$ is called \emph{Ahlfors $Q$-regular} for $Q>0$, if there exists a Borel measure $\mu$ and a constant $C\geq 1$ such that for any $x\in X$ and $0<r\leq \diam (X)$, 
\[ \frac1C r^Q\leq \mu (\overline B(x,r))\leq C r^Q,\]
Two metric space $(X,d_X)$ and $(Y,d_Y)$ are \emph{quasisymmetriclly equivalent}  if there are homeomorphisms  $f\:X\ra Y$ and $\eta\: [0,\infty)\ra[0,\infty)$ such that for all $x,y,z\in X$ with $x\not=z$, we have
\[\frac{d_Y(f(x),f(y))}{d_Y(f(x),f(z))}\leq \eta\left(\frac{d_X(x,y)}{d_X(x,z)}\right).\]

We have a natural metric on $\widehat{\C}=\C\cup \{\infty\}$ by stereographic projection, called the \emph{chordal metric}, defined by
\[\delta(z,w)=\frac{2|z-w|}{\sqrt{1+|z|^2}\sqrt{1+|w|^2}}, \]
\[\delta(z,\infty)= \delta(\infty,z)=\frac2{\sqrt{1+|z|^2}}\]and
\[\delta(\infty,\infty)=0\]
for $z,w\in \C$.

\begin{pro} \label{qstosphere}
If we let $d$ be a visual metric that we get under the assumption of Proposition \ref{existenceofvisualmetric}, then $(\S^2,d)$ is Ahlfors $2$-regular and quasisymmetrically equivalent to the Riemann sphere $\widehat{\C}$.
\end{pro}

\begin{proof}
Proposition 19.10 in \cite{BMExpanding} states that for an expanding Thurston map $f\:\S^2\ra \S^2$ without periodic critical points, if $d$ is a visual metric with expansion factor $\Lambda$, then $(\S^2,d)$ is Ahlfors $Q$-regular with
\[ Q=\frac{\log (\deg (f))}{\log \Lambda}.\] Since our $\Lambda=\deg(f)^{1/2}$, the metric space $(\S^2,d)$ is Ahlfors $2$-regular. Proposition 16.3 (iii) in \cite{BMExpanding} states that $\S^2$, with a visual metric $d$ for $f$, is linearly locally connected. Now our proposition follows immediately from Theorem 1.1 in \cite{BKQuasisymmetric}, which states that for a metric space $X$ homeomorphic to $\S^2$, if $X$ is linearly locally connected and Ahlfors $2$-regular, then $X$ is quasisymmetrically equivalent to the Riemann sphere $\widehat{\C}$.
\end{proof}

Theorem 1.7 in \cite{BMExpanding} states that:
\begin{thm}[Bonk-Meyer 2010]
For an expanding Thurston map with visual metric $d$, $(\S^2,d)$ is quasisymmetrically equivalent to the Riemann sphere $\widehat \C$ if and only if $f$ is topologically conjugate to a rational map.
\end{thm}

By this theorem and Proposition \ref{qstosphere}, there exists a rational map $R\: \S^2\ra \S^2$ and a homeomorphism $\phi$ such that $\phi \circ f=R\circ \phi$. See the diagram below:
\[\begin{CD}
\S^2 @>\phi>> \widehat{\C}\\
@VVf V @VV R V\\
\S^2 @>\phi>> \widehat{\C}
\end{CD}\]
Since $\phi$ is a homeomorphism,  $\deg(R)=\deg(f)$. The Jordan curve $\mathcal C'=\phi(\mathcal C)$ contains all the post-critical points of $R$, where $\mathcal{C}\subset \S^2$ is a Jordan curve containing $\post(f)$. Also, $X$ is an $n$-tile in the cell decomposition induced by $(f,\mathcal C)$ if and only if $\phi(X)$ is an $n$-tile in the cell decomposition induced by $(R,\mathcal C')$. So the minimal numbers of $n$-tiles needed to join opposite sides of the Jordan curves $\mathcal C$ and $\mathcal C'$ are the same, i.e., $D_n(R,\mathcal C')=D_n(f,\mathcal C)$. By Proposition \ref{existenceofvisualmetric} and Proposition \ref{qstosphere}, there exists a visual metric $d_R$ on $\widehat{\C}$ with expansion factor $\Lambda_0=(\deg(R))^{1/2}$ and $(\widehat{\C},d_R)$ is Ahlfors 2-regular. In addition, with the metric \[d_R(\phi(x),\phi(y))=d(x,y),\] $\phi$ is an isometry since $m_{f,\mathcal{C}}(x,y)=m_{R,\mathcal{C'}}(\phi(x),\phi(y))$. Hence, we have the following:

\begin{pro}\label{smallsum}
Let $f\:\S^2\ra \S^2$ be an expanding Thurston map with no periodic critical points. If there exists $c>0$ such that $D_n\geq c(\deg f)^{n/2}$ for all $n>0$, then $f$ is topologically conjugate to a rational map $R\:\widehat{\C}\ra \widehat{\C}$. In addition, there is a visual metric $d$ on $\widehat{\C}$ for $R$ with expansion factor $\Lambda=(\deg(R))^{1/2}$ such that $(\widehat{\C},d)$ is Ahlfors 2-regular.
\end{pro}

Corollary 18.4 in \cite{BMExpanding} says that:
\begin{lem} \label{chordalmetirc}
If $d$ is a visual metric for an expanding rational Thurston map $R$, then
the identity map ${\rm id} \: (\S^2,d)\ra (\S^2,\delta)$ is a quasisymmetry, where $\delta$ is the chordal metric.
\end{lem}

To state our next lemma, let us recall some definitions on metric spaces. We refer the reader to \cite{HKQuasiconformal} for more details. Given a real-valued function $u$ on a metric space $X$, a Borel function $\rho\: X\ra [0,\infty]$ is said to be an \emph{upper gradient} of $u$ if
\[|u(x)-u(y)|\leq \int_{\gamma}\rho\, d s\]for each rectifiable curve $\gamma$ joining $x$ and $y$ in $X$. If $u$ is a smooth function on $\R^n$, then its gradient $|\nabla u|$ is an upper gradient. We say that a metric space $X$ equipped with a (Borel) measure $\mu$ admits a \emph{$(1,p)$-Poincar\'e inequality} for $p\geq 1$, if there are constants $0<\lambda \leq 1$ and $C\geq 1$ such that for all balls $B$ in $X$, for all bounded continuous functions $u$ on $B$, and for all upper gradients $\rho$ of $u$ on $B$, we have that
\[ \frac1{\mu(\lambda B)}\int_{\lambda B}|u-u_{\lambda B}|\,d\mu \leq C(\diam B)\left(\frac1{\mu(B)}\int_B\rho^p\,d\mu\right)^{1/p}, \]
where $\lambda B$ is a scaling of the ball $B$ by $\lambda$ and
\[u_{\lambda B}=\frac1{ \mu(\lambda B)}\int_{\lambda B}u\,d\mu.\] Corollary 7.13 in \cite{HKQuasiconformal} states that:
\begin{thm}[Heinonen-Koskela 1998] \label{absolutecontinuous}
Let $X$ and $Y$ be two locally compact $Q$-regular spaces,
where $X$ satisfies a $(1,p)$-Poincar\'e inequality for $p<Q$. If $g$ is a quasisymmetric map from $X$ to $Y$, then $g$ and its inverse are absolutely continuous with respect to the Hausdorff $Q$-measure (of each individual space).
\end{thm}

To formulate the next theorem, we call a metric space $X$ \emph{linearly locally contractible} if there is a $C\geq 1$ so that, for each $x\in X$ and $R<C^{-1}\diam (X)$, the ball $B(x,R)$ can be contracted to a point in $B(x,CR)$. Theorem 6.11 in \cite{HKQuasiconformal} states that:
\begin{thm}[Heinonen-Koskela 1998]  \label{poincareinequality}
Let $X$ be a connected and $n$-regular metric space that is also an orientable $n$-manifold, with $n\geq 2$. If $X$ is linearly locally contractible,
then $X$ admits a $(1,p)$-Poincar\'e inequality for all $p\geq 1$.
\end{thm}
By the previous theorem, the Riemann sphere with the chordal metric satisfies a $(1,p)$-Poincar\'e inequality.

\begin{thm}\label{suff}
Let $f\:\S^2\ra \S^2$ be an expanding Thurston map with no periodic critical points. If there exists $c>0$ such that $D_n\geq c(\deg f)^{n/2}$ for all $n>0$, then $f$ is topologically conjugate to a Latt\`es map.
\end{thm}

\begin{proof}
By Proposition \ref{smallsum}, there exists a rational function $R$ conjugate to $f$, and $R$ has a visual metric $d$ with expansion factor $\Lambda=\deg(f)^{1/2}$ such that $(\S^2,d)$ is Ahlfors $2$-regular. Applying Lemma \ref{chordalmetirc} to the rational map $R$, the identity map ${\rm id} \: (\widehat{\C},d)\ra (\widehat{\C},\delta)$ is a quasisymmetry, where $\delta$ is the chordal metric.

The standard Riemann sphere with chordal metric $(\widehat{\C},\delta)$ satisfies a $(1,1)$-Poincar\'e inequality, with $p=1$ and $Q=2$ by Theorem \ref{poincareinequality}. By Lemma \ref{absolutecontinuous}, the (normalized) Hausdorff measure $H_{d}$ of the metric $d$ and the (normalized) Hausdorff measure $H_{\delta}$ of the metric $\delta$ are mutually absolutely continuous with each other. This implies that a set $E\subset \S^2$ has full measure under $H_{\delta}$ if and only if $E$ has full measure under $H_{d}$.

The dimension of the Lebesgue measure $\Delta$ (i.e., the normalized spherical measure of $\widehat\C$) with respect to the metric $d$ is
\begin{eqnarray*}
\dim (\Delta, d) &=& \inf\{\dim_{H_d}(E)\: \Delta(E)=1\} \\
                  &=& \inf\{\dim_{H_d}(E)\: H_{\delta}(E)=1\}\\
                  &=& \inf\{\dim_{H_d}(E)\: H_{d}(E)=1\}=2.
\end{eqnarray*}
Theorem \ref{dim} says that the dimension $\dim (\Delta, d)$ of the Lebesgue measure $\Delta$ with respect to the metric $d$ is equal to 2 if and only if $R$ is a Latt\`es map. Hence, $R$ is a Latt\`es map and $f$ is topologically conjugate to a Latt\`es map.
\end{proof}

\begin{eg} \label{latteseg3}
Recall the Latt\`es-type maps $f$ in Example \ref{latteseg} and $g$ in Example \ref{nonlatteseg}. Let Jordan curves $\mathcal C$ and $\mathcal C'$ be the same as described in Example \ref{latteseg2}. Then
\[D_n(f,\mathcal C)=2^n=\deg(f)^{n/2}\] and
\[D_n(g,\mathcal C')=2^n< 6^{n/2}=\deg(g)^{n/2}\]
for all $n>0$. By Theorem \ref{suff}, the map $f$ is topologically conjugate to a Latt\`es map while $g$ is not.
\end{eg}

In the proof of Theorem \ref{suff}, we also proved that
\begin{cor} \label{visuallattes}
Let $f\:\S^2\ra \S^2$ be an expanding Thurston map with no periodic critical points. If there exists a visual metric on $f$ with expansion factor $\Lambda=\deg(f)^{1/2}$,  then $f$ is topologically conjugate to a Latt\`es map.
\end{cor}

\newpage

\section{Conclusion}

\noindent
We get the following topological characterization of Latt\`es maps:

\begin{thm} \label{main}
A map $f\:\S^2\ra \S^2$ is topologically conjugate to a Latt\`es map if and only if the following conditions hold:
\begin{itemize}
  \item $f$ is an expanding Thurston map;
  \item $f$ has no periodic critical points;
  \item there exists $c>0$, such that $D_n\geq c(\deg f)^{n/2}$ for all $n>0$.
\end{itemize}
\end{thm}

\begin{proof}
Since all three conditions are preserved under topological conjugacy, we only need to check them for Latt\`es maps. If $f$ is a Latt\`es map, then $f$ is an expanding Thurston map without periodic critical points. In addition, by
Proposition \ref{opinequality}, there exists $c>0$, such that $D_n\geq c(\deg f)^{n/2}$ for all $n>0$.

The sufficiency of the three conditions follows from Theorem \ref{suff}.
\end{proof}

\begin{cor} \label{cormain}
A map $f\:\S^2\ra \S^2$ is topologically conjugate to a Latt\`es map if and only if the followings conditions hold:
\begin{itemize}
  \item $f$ is an expanding Thurston map;
  \item $f$ has no periodic critical points;
  \item there exists a visual metric on $\S^2$ with respect to $f$ with expansion factor $\Lambda=\deg(f)^{1/2}$.
\end{itemize}
\end{cor}

\begin{proof}
The sufficiency of these conditions follows directly from Corollary \ref{visuallattes}.

If $f$ is topologically conjugate to a Latt\`es map, then $f$ satisfies the three conditions in Theorem \ref{main}. By Proposition \ref{existenceofvisualmetric}, there exists a visual metric on $\S^2$ with expansion factor $\Lambda=\deg(f)^{1/2}$.
\end{proof}

\newpage

\bibliographystyle{amsplain}
\bibliography{qian}

\end{document}